\documentclass[11pt]{amsart}
\newif\ifPDF
\ifx\pdfoutput\undefined\PDFfalse
\else \ifnum \pdfoutput > 0 \PDFtrue
        \else \PDFfalse
        \fi
\fi

\ifPDF
  \usepackage[pdftex]{xcolor, graphicx}
  \usepackage[colorlinks=true,linkcolor=blue, citecolor=blue]{hyperref}%

\usepackage[centertags]{amsmath}
\usepackage{amsfonts}
\usepackage{mathrsfs}
\usepackage{textcomp}
\usepackage{amssymb}
\usepackage{amsthm}
\usepackage{newlfont}
\usepackage[all]{xy}

\else
  \usepackage{color}
  \usepackage[dvips]{graphicx}
  \usepackage[dvips]{hyperref}
\fi

\newtheorem{theorem}{Theorem}[section]
\newtheorem{corollary}[theorem]{Corollary}
\newtheorem{lemma}[theorem]{Lemma}
\newtheorem{proposition}[theorem]{Proposition}
\theoremstyle{definition}
\newtheorem{definition}[theorem]{Definition}
\newtheorem{remark}[theorem]{Remark}
\newtheorem{example}[theorem]{Example}
\numberwithin{equation}{section}

\theoremstyle{definition}

\newcommand{\ol}{\overline}
\newcommand{\wt}{\widetilde}
\newcommand{\N}{\mathbb N}

\newcommand{\R}{\mathbb R}

\newcommand{\F}{\mathcal F}
\newcommand{\B}{\mathcal B}

\newcommand{\s}{\sigma}
\newcommand{\sB}{{\sigma^{-1}(\B)}}
\newcommand{\be}{\begin{equation}}
\newcommand{\ee}{\end{equation}}
\newcommand{\ba}{\begin{aligned}}
\newcommand{\ea}{\end{aligned}}
\newcommand{\sms}{(X, \mathcal B, \mu)}
\newcommand{\FXB}{{\mathcal F(X, \B)}}
\def\csi1{\circ\sigma^{-1}}
\newcommand{\sqva}{\sqrt{\varphi}} 
\newcommand{\ignore}[1]{}

\def\va{\varphi}

\def\csi1{\circ\sigma^{-1}}
\def\cs{\circ\sigma}

\def\ol{\overline}
\def\wt{\widetilde}
\def\wh{\widehat}

\def\mc{\mathcal}

\usepackage{bbm}

\begin{document}

\title[MRS, endomorphisms, and Cuntz relations] 
{Measurable multiresolution systems, endomorphisms,
and representations of \\ Cuntz relations}


\author{Sergey Bezuglyi and Palle E.T. Jorgensen}
\address{University of Iowa, Iowa City, Iowa, USA}
\email{sergii-bezuglyi@uiowa.edu}
\email{palle-jorgensen@uiowa.edu}

\thanks{}
\date{\today}
\dedicatory{}
\commby{}


\begin{abstract}

The purpose of this paper is to present new classes 
of function systems as part of multiresolution analyses. 
Our approach is representation theoretic, and it 
makes use of generalized multiresolution function systems 
(MRSs). 
It further entails new ideas from measurable 
endomorphisms-dynamics. Our results yield applications that 
are not amenable to more traditional techniques used on metric spaces. 
As the main tool in our approach, we make precise new 
classes 
of generalized MRSs which arise directly from a dynamical 
theory approach to the study of surjective endomorphisms on  measure spaces. In particular, we give the necessary and sufficient conditions for a family of functions to define generators 
of Cuntz relations. We find an explicit description of the 
set of generalized wavelet filters. 
Our results are motivated in part by 
analyses of sub-band filters in signal/image processing. 
But our paper goes further, and it applies to such 
wider contexts as measurable dynamical systems, and complex 
dynamics.

A unifying theme in our results is a new analysis of 
endomorphisms in 
general measure space, and its connection to 
multi-resolutions, to 
representation theory, and generalized wavelet systems.

\end{abstract}

\maketitle

\tableofcontents

\setcounter{tocdepth}{1}

\section{Introduction}

The present paper is focused on relationships between the 
following objects: surjective endomorphisms of a measure 
space $\sms$, 
transfer operators in $L^2\sms$, generalized wavelet 
filters,  Markovian functions, and 
representations of the Cuntz relations.

Our analysis of transformations of a measure space, 
associated transfer operators, and
representations of the Cuntz algebras derive 
from a new analysis of endomorphisms in general measure 
space, presented here. While various special cases of 
endomorphisms have been studied in earlier work, we call 
attention here to three new elements: (i)  our context 
is that of the most general transformations of a
measure spaces; (ii) we 
introduce a new harmonic analysis into the problem via 
representations of Cuntz algebras; and (iii) our 
representations of Cuntz algebras for the purpose arise 
directly from the endomorphism at hand, $\sigma$, and a 
choice of a quasi-invariant measure $\mu$. From this, we 
then identify a new construction of an infinite-dimensional 
manifold  $\mathfrak{M}$ of generalized wavelet filters, and 
a transitive action of a canonical group $\mc G$, acting on 
$\mathfrak M$, and depending only on the given pair 
$\sigma$ (endomorphisms), and $\mu$ (measure). The 
representations of the particular Cuntz algebra (depending 
on $\sigma$) in turn define endomorphisms of $B(L^2(\mu))$, 
the $C^*$-algebra of all bounded operators in $L^2(\mu))$, 
and with each non-commutative endomorphism extending the 
initial endomorphism $\sigma$.

We present a new class of function systems $\mathfrak M$ 
as part of 
multiresolution analyses (MRAs), motivated in part by 
\cite{AlpayColombo2022}, \cite{PickloRyan2022},
\cite{ChristoffersonDutkay2022}. 
Our approach is based on methods of the representation 
theory, 
and it makes use of generalized iterated function systems 
related to a surjective endomorphism of a measure 
space. For background on the representation theory, we 
refer the readers to \cite{Jorgensen2018},
\cite{DutkayJorgensenSilvestrov2012}, 
\cite{JorgensenSong2018}.  
Our results for these generalized MRSs further entail new 
ideas from measurable dynamics, see e.g.,  
\cite{AlpayJorgensenLewkowicz2017},
\cite{AlpayJorgensenLewkowicz2018},
\cite{Andrianov2022}, \cite{BhatDar2022}. 
While our general focus is on branching systems in the 
measurable category, our applications are not amenable to 
more traditional metric techniques such as e.g., 
\cite{FengSimon2022}, \cite{Roychowdhury-2_2022}, 
\cite{MedhiViswanathan2023}.
We turn to new classes of generalized iterated function
systems which arise 
directly from a more general dynamical theory approach via 
a systematic study of endomorphisms in measure spaces. We 
are motivated in part by analyses of sub-band filters in 
signal/image processing, see e.g., 
\cite{DutkayJorgensen_2007},  
\cite{AlpayJorgensenLewkowicz2018}, 
\cite{Baggett_et_al2010}, 
\cite{BaggettMerrillPackerRamsay2012}.
But our paper goes further, and it encompasses such wider 
contexts as measurable dynamical systems and complex 
dynamics. The corresponding literature on wavelet filters, 
representations of Cuntz algebras, iterated function 
systems, transfer operators, and other adjacent areas are
very extensive; we mention here the following sources where
the reader can find more details and alternative approaches:
\cite{BratteliJorgensen_1997}, \cite{BratteliJorgensen_1999} 
\cite{BratteliJorgensen(1997)}, 
\cite{JorgensenTian2019}, \cite{BratteliJorgensen1997},
\cite{Jorgensen1996}, 
\cite{Jorgensen2001}, \cite{DutkayJorgensen_2007}, 
\cite{DutkayJorgensen2014}, \cite{DutkayJorgensen_2014}, 
\cite{DutkayJorgensen2006}, \cite{DutkayJorgensen2014},
\cite{DutkayJorgensen2015}, \cite{JorgensenPaolucci2011},
\cite{JorgensenKornelsonShuman2007}, 
\cite{JorgenssenKornelsonShuman2011}, 
\cite{JorgensenTian2017}.

As noted in the papers cited above, the traditional 
approach to iterated function systems, or more generally to 
semi-branching function systems, the starting point is 
typically a fixed system of maps that can be shown to admit 
limits in the form of attractors and measures invariant 
with respect to iterated function systems. These 
constructions are typically based on metric considerations, 
and they play a big role in such diverse applications as 
(fractal) harmonic analysis, graph Laplacian, boundaries, 
and 
analysis of geometries which are given by classes of 
self-similarity. Cantor and Sierpinski constructions are 
cases in point. The corresponding IFS can be shown in turn 
admit realizations in shift dynamical systems.

Our present approach is the opposite: we begin with a 
consideration of endomorphisms in measure spaces (see the 
definition in Subsection 
\ref{ssect endom}),  and of associated measurable 
partitions (Subsection \ref{subsect Meas part}). 
With this as starting point, we then introduce several 
bounded operators which will define 
representations of a non-abelian $C^*$-algebras given by 
generators and relations, called Cuntz algebras, 
Sections \ref{sect End Operators} and \ref{sect q-inv}.
and 4. There are two advantages to this approach, (i) it 
allows for a much wider family of (generalized) MRSs, and 
(ii) it also offers new and direct tools for attacking the 
corresponding harmonic analysis questions. 
Finally, we give an explicit description of the set of
generalized wavelet filters, see Section 
\ref{sect wavelet filters}. 

Measurable transformations of a standard measure space 
$\sms$ is the central concept of the ergodic theory. 
Invertible transformations (automorphisms) and their 
properties have been extensively studied from various 
points of view. The study of non-invertible transformations
(endomorphisms) has been less popular than that of 
automorphisms although their role 
in dynamics and adjoint areas is extremely important. 
In particular, they are used in the construction of iterated 
function systems (IFSs), transfer operators, and wavelet 
filters. We refer here to several recent books dealing 
with endomorphisms and their applications in the operator 
theory \cite{Jorgensen2006(book)}, 
\cite{PrzytyckiUrbanski2010}, 
\cite{BezuglyiJorgensen2018(book)}, 
 \cite{Hawkins2021}, \cite{Eisner(book)2015},
\cite{Bruin(book)2022}, \cite{Urbanski(book)2022}.

Using an analysis of endomorphisms of a standard measure 
space, we associate with every endomorphism several bounded 
operators acting in $L^2\sms$.
Non-singular endomorphisms $\sigma$ of a measure space 
$\sms$ are naturally divided 
into two classes: if $\mu = \mu\csi1$, then $\s$ is called 
\textit{measure-preserving}; if $\mu \sim \mu\csi1$, 
then the measure $\mu$ is called \textit{quasi-invariant} 
with respect to $\s$. If additionally, $\mu\cs \sim \mu$ on 
$\sB$, then $\s$ is called 
\textit{forward quasi-invariant}. In this case, there are 
Borel functions $\va$, called \textit{Markovian} functions,
such that 
\be\label{eq:forward RN}
\int_X (f\cs) \va\; d\mu = \int_X f \; d\mu.
\ee
This relation is the basis for defining isometric operators: 
the \textit{composition operator}  
$S_\s :f \mapsto f\cs$ in the case of an invariant measure
$\mu$ and the \textit{weighted composition 
operator} $S_\va : f \mapsto \sqva (f\cs)$ for a
quasi-invariant measure $\mu$ (here $\va$ is Markovian).
These operators, together with transfer operators, will 
play key roles in our constructions. 

We formulate now our \textit{main results} and outline 
the paper's organization.  
In Section \ref{sect Prelim}, we define the main objects of
this paper. They are: a standard measure space, measurable 
partitions, canonical systems of measures, subjective 
endomorphisms, invariant and quasi-invariant measures, and
Markovian functions. All these notions are used in the next 
sections. Section \ref{sect Prelim} should be viewed as
a brief survey on endomorphisms and related notions. 
Section \ref{sect End Operators} is focusing on the study
of linear operators generated by an endomorphism of 
a measure space. We define abstract transfer operators 
acting on bounded Borel functions and consider their 
properties. If the operator $S_\s$ is considered in 
$L^2\sms$ where $\mu$ is $\s$-invariant, then $S^*_\s$ is 
a transfer operator which has interesting properties,
see Theorems \ref{thm R = S^*}, \ref{thm Phi in L2},
\ref{thm om_mu and om_nu}, and \ref{thm E vs E*}. 
In particular, the operator $S^*_\s$ coincides with 
a transfer operator $R_\s$ defined by the measurable 
partition into pre-images of $\s$. We prove also 
similar results for weighted composition operators.
Section \ref{sect q-inv} contains the principal theorems 
connecting wavelet filters with representations of the 
Cuntz relations. We recall that a family of isometries 
$\{T_i : i \in \Lambda\}$ defines the Cuntz relations if
$\sum_{i\in \Lambda} T_iT^*_i = \mathbb I$ and the 
projections $T_iT^*_i$ are mutually orthogonal for 
different indexes where $\Lambda$ is finite or countable. 
Let $\va$ be a Markovian function and $m$ a complex-valued 
function such that $S^*_\va(\sqva |m|^2) = \mathbbm 1$. 
Define $T_m (f) = m S_\va(f)$. Then $T_m$ an isometry in 
$L^2(\mu)$. We prove the following results (see Theorem 
\ref{thm main1}).

\begin{theorem} 
Let $\sigma$ be an onto 
endomorphism of $\sms$ where $\mu$ is quasi-invariant with 
respect to $\s$. 
Let $\{m_i : i \in \Lambda\}$ be a family of
complex-valued functions.  
The operators $\{T_{m_i}, i \in \Lambda$\} generate  a 
representation of the Cuntz algebra $\mathcal O_{|\Lambda|}$ 
if and only if 
$$
(i)\ \ S^*_\va (\sqva \; \ol m_j m_i) = \delta_{ij} 
\mathbbm 1,
$$
$$
(ii) \ \ \sum_{i\in \Lambda} m_i \mathbb E_\va( \ol m_i f) 
= f,
\ \ \ f \in L^2(\mu).
$$ 
\end{theorem}

Here $\mathbb E_\va = S_\va S^*_\va$ is the orthogonal 
projection from $L^2(\mu)$ onto a subspace $\mc H_\va$.

As a corollary, we have the following decomposition of
$L^2(\mu)$ in the case of $\s$-invariant measure $\mu$:
$$
L^2(\mu) =\bigoplus_{i\in \Lambda} \ m_i L^2(X, \sB, \mu).
$$

In Section \ref{sect wavelet filters}, we focus on 
finding a description of the families of functions
$\underline m = (m_i) \in \mathfrak M_\va$ satisfying the 
above theorem.
Let $\mc G$ be the group of Borel functions with values in 
the unitary operators on $\ell^2(\Lambda)$. 
Then we prove the following:

\begin{theorem} [Theorem \ref{thm action}] 
(1) The set $\mathfrak M_\va$ is isomorphic (as a set) to the
loop group $\mc G$. 

(2) For every element $G = (g_{ij})$ of the loop group, 
there exists a wavelet filter $\underline m$ such that 
$g_{ij} = S^*_{\va}(\sqva\; m_i \ol m_j)$.
\end{theorem}

Let functions $(m_i : i \in \Lambda)$ satisfy the property.
\be\label{1.2}
S_\s^*(|m_i|^2) = \mathbbm 1
\ee
and $S_i(f) = m_i (f\cs)$ is an isometry on $L^2(\mu)$.

\begin{theorem} [Theorem \ref{thm main 2}]
Let $(m_i : i\in \Lambda)$ be a set of cyclic vectors 
for a representation of $L^\infty(X, \sB, \mu)$ on 
$L^2(\mu)$ that 
satisfies   \eqref{1.2}. 
Then $\underline m = (m_i)$ is a wavelet 
filter if and only if $\sum_{i \in \Lambda} S^*_iS_i = 
\mathbb I$. In other words, $\underline m \in \mathfrak M$ if
and only if the operators $S_i$ are the generators of a 
representation of the Cuntz algebra $\mathcal O_{|\Lambda|}$.
\end{theorem}

\section{Basics on endomorphisms}\label{sect Prelim}

In this section, we give  basic definitions and facts 
from the theory of endomorphisms of a standard measure space 
$(X, \B, \mu)$. The notion of an endomorphism is one 
of the central concepts of ergodic theory; the foundations 
and more advanced results on endomorphisms can be found in 
some pioneering papers in ergodic theory and in more recent 
papers and books, see e.g.,  \cite{Rohlin1961}, 
\cite{Hawkins1994}, \cite{CornfeldFominSinai1982}, 
\cite{HawkinsSilva1991}, \cite{Beneteau1996},  
\cite{BruinHawkins2009},  \cite{PrzytyckiUrbanski2010}, 
\cite{Hawkins2021} and the papers
cited therein. The study of endomorphisms is mostly based on 
the notion of a measurable partition of a measure space and 
associated subalgebras of Borel sets. 
A systematic study of measurable dynamical systems based on 
applications of measurable partitions was initiated by Rokhlin in 
\cite{Rohlin_1949, Rohlin1949, Rohlin1961}.

\subsection{Endomorphisms of a measure 
space}\label{ssect endom} 

We begin with the definitions of the main objects considered in 
the paper. 

Let $(X, \B)$ be a \textit{standard Borel space}, i.e., 
$(X, \B)$ is Borel isomorphic to a Polish space with the 
sigma-algebra of Borel sets. 
If $\mu$ is a non-atomic Borel positive measure on $(X, \B)$, 
then $\sms$ is called a \textit{standard Borel space}.  

By an \textit{endomorphism} $\sigma$ of $(X, \B, \mu)$ (or
$(X, \B)$) we mean a measurable  (or Borel) map of $X$ 
\textit{onto} itself ($\s$ is subjective). We will discuss  
various properties of endomorphisms below. In particular, $\s$
defines a partition of $X$ into subsets $\{\s^{-1}(x): x \in X\}$.
Depending on the cardinality of the sets $\s^{-1}(x)$, we call
$\s$ either finite-to-one, or countable-to-one, or 
continuum-to-one. Without loss of generality, we can assume
that the cardinality $|\s^{-1}(x)|$ is constant. The collection
of sets $\sigma^{-1}(A), A \in \B$, forms a $\s$-subalgebra of 
Borel sets which is denoted $\sB$ (below it will be also 
denoted  by $\mc A$ to shorten formulas). 
In general, the set $\sigma(A)$ is not Borel for every $A\in \B$, 
but if $\sigma$ is at most countable-to-one then $\sigma(A)$  is 
automatically Borel. 

Let $End(X, \B)$ denote the set (semigroup) of all 
\textit{surjective endomorphisms} of a standard Borel space 
$(X, \B)$. By $M_1(X)$, we denote the set of all 
\textit{Borel probability non-atomic measures}. An element 
of $M_1(X)$ will be simply called a measure in the paper. 
For $\sigma \in End(X, \B)$ and $\mu \in M_1(X)$, 
we define the measure $\mu\csi1$ where $\mu\csi1(A) := 
\mu(\sigma^{-1}(A))$. Then, the map $ \mu \mapsto \mu\csi1$ 
defines an action of $End(X, \B)$ on $M_1(X)$. We will be 
interested in the following cases: (i) the measure $\mu\csi1$ is 
equivalent to $\mu$, and (ii) $\mu$ is  $\s$-invariant. 
In case (i), we say that an endomorphism  $\s$ is 
\textit{non-singular}, i.e., 
\be\label{eq quai-inv meas}
\mu(A) = 0 \  \Longleftrightarrow \ \mu(\sigma^{-1}(A)) = 
0,\ \ \ \forall A\in \B.
\ee
In other words, the measure $\mu$ is called (backward) 
$\sigma$-\textit{quasi-invariant}, in symbols, $\mu\cs^{-1} \sim 
\mu$. We use the notation $End(X, \B, \mu)$ to denote the
semigroup of all surjective endomorphisms $\sigma$ such that 
$\mu$ is quasi-invariant with respect to $\sigma$.
For a fixed $\mu$,  the set $End(X, \B, \mu)$ contains the 
sub-semigroup $End_\mu$ which consists of the endomorphisms  
preserving $\mu$, i.e., $\mu(A) = \mu\csi1(A)$. 
The set $End_\mu$ can be viewed as the stabilizer of the 
action of $End(X, \B)$ on $M_1(X)$ at  $\mu$.

We will need also the notion of a 
\textit{(forward)  quasi-invariant measure}  $\mu$. 
This means that for every $\mu$-measurable  set 
$A$, the set  $\sigma(A)$ is measurable and $\mu(A) = 0 \ 
  \Longleftrightarrow  \ \mu(\sigma(A)) = 0$. 

 
\begin{lemma}
\label{lem existence of q-inv measure for  sigma}
Let $\sigma$ be a surjective endomorphism of a standard 
Borel space $(X, \B)$. Then $M_1(X)$ always contains a 
$\sigma$-quasi-invariant measure $\mu$.  
\end{lemma}

\begin{proof}
Every endomorphism  $\sigma$ generates a countable Borel 
equivalence relation $E(\sigma)$ whose classes are the 
orbits of $\sigma$. By definition, $(x, y) \in E(\sigma)$ 
if there exist $m, n \in \N_0$ such that 
$\sigma^n(x)  = \sigma^m(y)$.   
Quasi-invariant measures for $\sigma$ coincide
with quasi-invariant measures for $E(\sigma)$. Then we can 
use \cite[Proposition 3.1]{DoughertyJacksonKechris1994} 
where the existence of $E(\sigma)$-quasi-invariant measures 
was proved.
\end{proof}

\begin{example}\label{ex basic ex}
Let $\sms = \prod_{i\in \N} (X_i, \B_i, \mu_i)$. Define the left
shift $\s$ on $X$: $\s(x_1, x_2, x_3, \cdots) = 
(x_2, x_3, \cdots)$. Then $\s$ is an endomorphism. If $|X_i| = N$
for all $i$, then $\s$ is $N$-to-one. Clearly, this construction 
can give other types of endomorphisms classified by the 
cardinality of $\s^{-1}(x)$. If all $(X_i, \B_i, \mu_i) = 
(Y, \mc C, \nu)$ are the same and $\mu$ is the product-measure 
$\otimes_i \nu$, then $\mu$ is $\s$-invariant. 
The measure $\mu$ can be quasi-invariant with respect to 
the left shift $\s$ if we use different measures $\mu_i$. 
More details are in \cite{HawkinsSilva1991}, 
\cite{DajaniHawkins1994}, and other papers of these authors.  
\end{example}

\subsection{Measurable partition}\label{subsect Meas part}

We refer to \cite{Rohlin1949} (or 
\cite{CornfeldFominSinai1982}) for the definition of a 
measurable partition.

Let $\xi = \{C_{\alpha} : \alpha \in I\}$ be a 
\textit{partition}  of a standard
probability  measure space $(X, \mathcal B, \mu)$ such that 
$C_\alpha \in \B$ (the index set $I$ can be either  
countable or uncountable; we focus on the case of an 
uncountable set). A Borel set $A$ of the form 
$\bigcup_{\alpha \in I'} C_\alpha$, $I' \subset I$, is 
called a  \emph{$\xi$-set}. 
 Let   $\B(\xi)$ be the sigma-algebra generated by all 
 $\xi$-sets. 
   
A partition $\xi$ is called \emph{measurable} if  
$\mathcal B(\xi)$ contains a countable subset $\{D_j\}$ of 
$\xi$-sets that separates any two elements $C, C'$ of $\xi$. 
This means that  there exists $D_i$ such that either 
$C \subset D_i$ and $C' \subset 
X \setminus D_i$ or $C' \subset D_i$ and $C \subset 
X \setminus D_i$. Let $\pi$ be the natural projection from 
$X$ to $X/\xi$, i.e., $\pi(x) = C_x$ where $C_x$ is the 
element of $\xi$ containing $x$. Using the projection  
$\pi : X \to X/\xi$, one can define a measure space 
$(X/\xi, \B/\xi, \mu_\xi)$ where $E \in  \B/\xi$ if and 
only if the $\xi$-set $\pi^{-1}(E) $ is in $\B$ and   
$\mu_\xi = \mu \circ \pi^{-1}$. 

The following result was proved in  \cite{Rohlin1949}.
\begin{lemma}
A partition $\xi$ of a standard measure space $\sms$ is measurable if and
 only if $(X/\xi,  \B/\xi, \mu_\xi)$ is a standard measure 
 space.    
\end{lemma}

It is said that a partition $\zeta$  \emph{refines} $\xi$
(in symbols, $\xi\prec \zeta $) if every element $C$ of 
$\xi$ is a $\zeta$-set. It turns out that every partition 
$\zeta$ has a \emph{measurable hull},
that is a measurable partition $\xi$ such that $\xi \prec 
\zeta$ and 
$\xi$ is a maximal measurable partition with this property.
If $\xi_\alpha$ is a family of measurable partitions, then
their product $\bigvee_\alpha \xi_\alpha$ is a measurable 
partition  $\xi$ which is uniquely determined by the 
conditions: (i) $\xi_\alpha
\prec \xi$ for all $\alpha$, and (ii) if $\eta$ is a 
measurable partition
such that $\xi_\alpha \prec \eta$, then $\xi \prec \eta$. 
Similarly, one defines the intersection 
$\bigwedge_\alpha\xi_\alpha$ of
measurable partitions. There is a one-to-one correspondence 
between the set of measurable partitions of a standard 
measure space $(X, \mathcal 
B, \mu)$ and  the set of complete sigma-subalgebras 
$\mathcal B'$ of $\mathcal   B$.

The role of measurable partitions becomes clear from Theorem
\ref{thm Rokhlin disintegration} given below. This famous 
result uses the notion of measure disintegration. 

\begin{definition}\label{def system cond measures}
For a  standard probability  measure space  
$(X, \mathcal B, \mu)$
and a measurable partition $\xi$ of $X$, it is said that a  
collection of measures $(\mu_C)_{C \in X/\xi}$ is a 
\emph{system of conditional measures}  with respect to 
$(X, \mathcal B, \mu)$ and $\xi$ if

(i) for each $C \in X/\xi$,  $\mu_C$ is a measure on the
sigma-algebra $\mathcal B_C := \mathcal B \cap C$ such that 
$(C, \mathcal B_C, \mu_C)$ is a standard probability measure 
space;

(ii) for any $B \in \mathcal B$, the function $C \mapsto 
\mu_C(B \cap C)$ is $\mu_\xi$-measurable;

(iii)  for any $B \in \mathcal B$,
\be\label{eq mu(B) integral}
\mu(B) = \int_{X/\xi} \mu_C(B \cap C)\; d\mu_\xi(C).
\ee
\end{definition}

Condition \eqref{eq mu(B) integral} can be rewritten in the 
equivalent form:
\be\label{eq_disint}
\int_X f(x)\; d\mu(x) = \int_{X/\xi} \left( \int_{C} f(y) \;
d\mu_C(y)\right) \; d\mu_\xi(C).
\ee

Measurable partitions are characterized by the following 
result.

\begin{theorem}[\cite{Rohlin1949}] 
\label{thm Rokhlin disintegration}
For any measurable partition $\xi$ of a standard probability 
measure space $(X, \mathcal B, \mu)$, there exists a unique 
system of conditional measures 
$(\mu_C)$. Conversely, if  $(\mu_C)_{C \in 
X/\xi}$ is a system of conditional measures with respect to 
$((X, \mathcal B, \mu), \xi)$, then $\xi$ is a measurable 
partition.
\end{theorem}

We apply Theorem \ref{thm Rokhlin disintegration}  
to the case of an endomorphism 
$\sigma \in End(X, \B, \mu)$. Let $\xi_\sigma$ be the 
measurable partition of $\sms$  into preimages 
$\sigma^{-1}(x) = C_x$ of points $x \in X$. Let $(\mu_C)$ be the 
system of conditional measures defined by $\xi_\sigma$. In the
case when $(X/\xi, \mu_\xi)$ is isomorphic to $(X, \mu)$ 
(for example, when $\s$ is the left shift or $\s : z \mapsto z^N$,
$z \in \mathbb T^1$), we see that relations 
\eqref{eq mu(B) integral} and \eqref{eq_disint} have the form 
\be\label{eq function disintegration for endo}
\int_X f(x)\; d\mu(x) = \int_X \left( \int_C f(y)\; 
d\mu_x(y)\right)\, d\mu(x). 
\ee
This decomposition is the key fact in our representation of the 
transfer operator generated by $\s$.

In most important cases, the disintegration of a measure is 
applied to probability (finite) measures. The problem of 
measure disintegration is discussed in many books and 
articles. We refer here to 
 \cite{Bogachev2007}, \cite{CornfeldFominSinai1982}, 
 \cite{Fabec1987, Fabec2000}, 
 \cite{Kechris1995}. The case of an infinite sigma-finite
measure was considered by several authors, see e.g.
   \cite{Simmons2012}. 

\begin{theorem}[\cite{Simmons2012}]\label{thm simmons}
Let $\sms$ and $(Y, \mathcal A, \nu)$ be standard measure 
spaces with sigma-finite measures, and suppose that $\pi : 
X \to Y$ is a measurable map. 
Let $\sms$ and $(Y, \mathcal A, \nu)$ be as above. 
Suppose that 
$\widehat \mu= \mu\circ \pi^{-1} \ll \nu$. Then there exists 
a unique system of conditional measures $(\nu_y)_{y \in Y}$ 
for $\mu$. For  $\nu$-a.e., $\nu_y$ is a sigma-finite 
measure.
\end{theorem}

\subsection{Radon-Nikodym derivatives and Markovian 
functions} \label{subsect R-N}

Suppose $\sigma \in End\sms$. Recall that $\mu\cs^{-1} \sim \mu$ 
in this case. Then we can define  the 
\textit{Radon-Nikodym derivative} $\rho_\mu(x):= 
\frac{d\mu\circ\sigma^{-1}}{d\mu}(x)$, which is a Borel 
function  such that
\be\label{eq RN rho}
 \int_{X} f(\sigma x) \; d\mu = \int_X f(x) 
 \rho_\mu(x)\;d\mu, 
 \quad f \in L^1(\mu).
 \ee
Setting $\rho_n(x) =  \frac{d\mu\circ\sigma^{-n}}{d\mu}(x)$, 
we obtain the Radon-Nikodym \textit{cocycle} satisfying the 
equation  $\rho_{n+m}(x) = \rho_m(\sigma^n(x)) \rho_n(x)$. 

\begin{remark}\label{rem rho for mu and nu}
Suppose that $\sigma \in End\sms$ and $\nu$ is a measure
equivalent to $\mu$, i.e., there exists a Borel  function 
$\xi$ such that $d\nu(x) = h(x) d\mu(x)$. Then $\sigma$ 
is also non-singular with respect to $\nu$, and  
$\rho_\nu(x) =  h(\sigma x)  \rho_\mu(x) h^{-1}(x)$.
\end{remark} 

Every $\sigma \in End(X, \B)$ defines a linear operator 
$S_\sigma$ called a composition operator on the space 
of bounded Borel functions:
$$
S_\sigma (f) = f\circ\sigma, \qquad f \in \mathcal F(X, \B).
$$
This operator is also known by the name of a \textit{Koopman 
operator} 
when it is considered in a $ L^2$ space. 

The map $A \mapsto \mu(\sigma(A)), A \in \B,$ 
defines a measure on the subalgebra $\sigma^{-1}(\B)$. If 
$\mu$ is a \textit{forward quasi-invariant} measure, then 
there exists a unique $\sigma^{-1}(\B)$-measurable function 
$\omega_\mu(x) = \frac{d\mu\cs}{d\mu}(x)$ such that 
\be\label{eq Markovian}
\int_{X} f(\sigma x) \omega_\mu(x)\; d\mu = \int_X f(x) 
\;d\mu, \quad f \in  L^1(\mu).
\ee
It can be deduced from the uniqueness of the Radon-Nikodym 
derivative that 
 $$
 \omega_\mu(x) = \rho_\mu(\sigma(x))^{-1}  
 $$ 
 when these functions are 
 considered as functions measurable with respect to 
 $\sigma^{-1}(\B)$.

The following statement is well-known and we omit its 
proof.

 \begin{lemma}\label{lem_isom}
(1) The composition operator $S_\sigma : L^2(X, \B, \mu) 
\to L^2(X, \sB,\mu)$ is an isometry if and only if $\mu 
\circ \s^{-1} = \mu$. 

(2) The  operator $S_\sigma$ on $L^2(\mu)$ is bounded if 
and only if there exists a constant $k > 0$ such that 
$$
 \frac{\mu(\sigma^{-1}(A))}{\mu(A)} \leq k, \quad A \in \B.
$$

(3) If $\mu$ is a forward quasi-invariant measure, then 
$$
T_\s : f \longmapsto \sqrt{\omega_\mu} (f \cs)
$$
is an isometry from $L^2(X, \B, \mu)$ onto $L^2(X, 
\sB, \mu)$.
\end{lemma}

\subsection{Properties of endomorphisms}
In this subsection, we collected the properties of endomorphisms 
of a measure space for the reader's convenience.  Here and 
below we implicitly use the 
$\mathrm{mod}\ 0$-convention which means that a property 
(formula, relation, etc) holds almost everywhere with 
respect to a fixed measure. 

\begin{definition}\label{def ergodic and exact}
Let $\sigma$ be a surjective endomorphism of $\sms$ with 
quasi-invariant measure $\mu$.  

(i) The endomorphism $\sigma$ is called 
\textit{conservative} if for any set $A$ of positive measure 
there exists $n >0$ 
such that $\mu(\sigma^n(A) \cap A) > 0$.

(ii) The endomorphism $\sigma$  is called  \textit{ergodic} 
 if whenever $A$ is $\sigma$-invariant, i.e., $
\sigma^{-1}(A) = A$,  then either $A$ or $X \setminus A$ is 
of measure zero.  Equivalently, $\sigma$ is ergodic if, for 
a bounded Borel function $f$,  the condition 
$f \cs = f$ implies that $f$ is a constant $\textrm{mod}\ 0$. 

(iii) Any endomorphism $\sigma \in End\sms$ generates the 
sequence of sub-algebras:
$$
\B \supset \sigma^{-1} (\B) \ \cdots \ \supset \sigma^{-i} 
(\B)\supset \ \cdots
$$
Then $\sigma \in End\sms$ is called \textit{exact} if 
 $$
 \B_\infty := \bigcap_{k \in \N} \sigma^{-k}(\B) = 
 \{\emptyset, X\} \  \mod 0.
$$

 \ignore{
(iv) The surjective endomorphism $\sigma$ of a probability 
 measure space $\sms$ is called  \textit{full} (or 
 l\textit{imsup full}) if  $\lim_{n\to \infty} 
 \mu(\sigma^n(B)) =1$ (or $\limsup_{n\to \infty} 
 \mu(\sigma^n(B))  = 1$) for every $B \in \B, \mu(B) > 0$.
 }
 
 (iv) A non-singular endomorphism $\sigma$ of $(X, \B, \mu)$ 
 is said to be $\mu$-\textit{recurrent} if for every non-
negative Borel  function $f$, the function 
$$
\Sigma(f) = \sum_{n\geq 0} f(\sigma^n(x)) \omega_n(x) 
$$
takes only the values 0 and  $\infty$ $\mu$-a.e. where 
 $\omega_n(x) = \dfrac{d\mu\cs^n}{d\mu}(x)$.
\end{definition}

In the next remark, we include several results illustrating the 
properties of endomorphisms given in Definition
\eqref{def ergodic and exact}. We use some results from 
\cite{Silva1988}, \cite{HawkinsSilva1991},  \cite{Hawkins1994}.

\begin{remark} 
(1) Every exact endomorphism is ergodic. There are 
examples of
ergodic endomorphisms which are not exact. As it is 
customary in ergodic theory, we can always assume, without 
loss of generality, that an endomorphism is ergodic. 

(2) There are examples of one-sided shifts ($n$-to-one 
endomorphisms) which are not exact.

(3)  We note that there are ergodic endomorphisms that are 
not conservative. 

(4) An endomorphism $\sigma$ is recurrent with respect to a  
finite measure $\mu$  if and only if $\sum_{n\geq 0} 
\omega_n(x) = \infty$.

(5) A $\mu$-recurrent endomorphism is conservative. 
\end{remark}

Since every surjective endomorphism defines an isometry
$S_\s$ (or $T_\s$), see Lemma \ref{lem_isom}, then
we can apply Wold's theorem to these objects.

\begin{theorem}[Wold's theorem]\label{thm Wold} 
Let $S$ be an isometric operator in a Hilbert space 
$\mc H$. Define 
$$
\mc H_{\infty} = \bigcap_n S^n \mc H, 
$$
and 
$$
\mc H_{shift} = N_{S^*} \oplus SN_{S^*} \oplus \cdots 
\oplus S^kN_{S^*} \oplus \cdots.
$$
Then the following statements hold.

(1) The space $\mc H$ is decomposed into the orthogonal 
direct sum
$$
\mc H = \mc H_\infty \oplus \mc H_{shift}.
$$

(2) 
$$
\mc H_{\infty} = \{ x \in \mc H : \| (S^*)^n x \| = \| x\|, 
\ \forall n \in \N\}.
$$

(3) The operator $S$ restricted on $\mc H_\infty$ is a 
unitary operator, and $S$ is a unilateral shift 
in the space $\mc H_{shift}$.
\end{theorem}

For $\s \in End(X, \B, \mu)$, define
$$
\B_\infty = \bigcap_{n =0}^\infty \sigma^{-n}(\B),
$$
and let $\mc A_\sigma = \{A \in \B : \sigma^{-1}(A) = A\}$
be the subalgebra of $\sigma$-invariant subsets of $X$. 

Let $\zeta$ be a partition of $\sms$ into orbits of 
$\s$ (recall that $x \stackrel{\zeta} \sim y$ if there are
$n,m$ such that $\s^n (x) = \s^{m}(y)$). Let $\eta$ be
the partition of $\sms$ such that 
$x \stackrel{\eta} \sim y$ if there is 
$n$ such that $\s^n (x) = \s^{n}(y)$. By $\zeta'$ and
$\eta'$ we denote the measurable halls of $\zeta$ and 
$\eta$, respectively.

If $\epsilon$ denotes the partition of $X$ into points, 
then we have the sequence of decreasing measurable 
partitions $\{\sigma^{-i}(\epsilon)\}_{i=0}^\infty$:
$$
\epsilon \succeq \sigma^{-1}(\epsilon )\succeq 
\sigma^{-2}(\epsilon) \cdots.
$$
As shown in \cite{Rohlin1961}, the following results hold
 $$
\zeta' \preceq \eta', \qquad \ \ \ \eta' = \bigwedge_n 
\sigma^{-n}(\epsilon)
$$
and
$$
\mc A(\zeta') = \mc A_\sigma, \qquad \ \ \ \mc A(\eta') = 
\B_\infty.
$$

In particular,  $\sigma$ is ergodic if the partition 
$\zeta'$ is trivial, and $\sigma$ is exact if the 
partition $\eta'$ is trivial.

Since $\eta' $ is a measurable partition, we can define 
the quotient measure space $(Y, \nu) = (X/\eta', \B/\eta', 
\mu_{\eta'})$ where $\B/\eta' = \B_\infty$.

The following result is deduced from Wold's theorem,
see details in \cite{BezuglyiJorgensen2018(book)}. 

\begin{corollary}
(1) Let $\pi: X \to Y$ be the natural projection. Then 
there exists a measure-preserving  automorphism 
 $\wt \sigma : (Y, \nu) \to 
(Y, \nu)$ such that $\wt \sigma$  is an automorphic factor 
of $\sigma$, i.e.,
$$
\wt \sigma \circ \pi = \pi \circ \sigma.
$$

(2) Let $S_\s : f \to f\circ \sigma$ be the isometry on 
$\mc H = L^2(\mu)$. Then, in the Wold 
decomposition $\mc H = \mc H_\infty \oplus 
\mc H_\infty^\bot$ for $S_\s$, we have
$$
\mc H_\infty = L^2(Y, \nu),
$$
and the restriction of $S_\s$ to $\mc H_\infty$ 
corresponds to the unitary operator $U$ defined by 
$\wt\sigma$, 
$U(f) = f\circ \wt \sigma$.
\end{corollary}

It turns out that every non-singular endomorphism is 
a factor of an invertible dynamical system. It is said that 
an automorphism $T \in Aut (Y, \mathcal C,
\nu)$ is a \textit{natural extension} of an endomorphism 
$\sigma \in 
End\sms$ if there exists a map $\tau : (Y, \mathcal C,\nu) 
\to \sms$ such that 

(i) $\mathcal C = \bigvee_{n =0}^\infty T^n(\tau^{-1} \B), \ 
\ \mod 0$,

(ii) there exists a measure $\nu' \sim \nu$ such that 
$\omega_{\nu'} = \omega_\mu  \circ \tau$.

We recall an important fact saying that every non-singular 
endomorphism $\sigma$ of $\sms$ admits a a natural 
extension, see \cite{Rohlin_1949}, \cite{Silva1988},  
\cite{Beneteau1996}, and Example \ref{ex solenoid}.
  
\begin{example}\label{ex solenoid}
Let $\sigma$ be an onto endomorphism of a standard Borel 
space $(X, \B)$.
Define the set $\wh X$ as follows: $\wh X$ is a subset of 
$X \times X \times X  \times \cdots $ such that 
$$
\wh x = (x_i)_{i \geq 0} \in \wh X  \ \Longleftrightarrow \ 
\sigma(x_{i+1}) = x_i \ \ \forall i\geq 0.
$$
The set $\wh X$ is often called the \textit{solenoid} 
constructed by $(X, \sigma)$ and denoted 
$Sol_{(X, \sigma)}$. Note that $\wh X$ is a closed subset
in the product space $X \times X \times  \cdots $. 
Since $\wh X$ is Borel,
it  inherits the Borel structure $\wh \B$ from the product 
space.   

Define the map $\wh \sigma : \wh X \to \wh X$  by setting
$$
\wh \sigma (x_0, x_1, x_2, \dots ) = (\sigma(x_0), x_0, x_1, 
\dots).
$$
Then one can easily verify that $\wh\sigma$ is a one-to-one 
Borel map of $\wh X$ onto itself, and the shift 
$$
\tau :  (x_0, x_1, x_2, \dots ) \mapsto  (x_1, x_2, x_3, 
\dots ) 
$$
is inverse to $\wh \sigma$, $\tau = \wh \sigma^{-1}$.

Let $\pi_n : \wh X \to X$  be the projection from $\wh X$ 
onto the $n$-th coordinate: for $\wh x = (x_n)$, set 
$\pi_n(\wh x) = x_n$, $n \geq 0$. 
Then  $\pi_n$ can be extended to a map $f \mapsto 
f\circ\pi_n$ from $\F(X, \B)$ to $\F(\wh X,\wh\B)$. It 
follows from the above  definitions that $\pi_{n+1} 
\wh\sigma(\wh x) = \pi_n (\wh x)$, and 
$\pi_0$ is a factor  map from $(\wh X, \wh\sigma)$ to 
$(X, \sigma)$, i.e.,
 $$
 \pi_0 \wh\sigma = \sigma \pi_0. 
 $$

\begin{proposition}
Let $\mu$ be a Borel continuous measure on a standard Borel 
space  $(X, \B)$ which is quasi-invariant with respect to  
an endomorphism $\sigma$. 
Let the measure $\mathbb P $ on $(\wh X, \wh B)$ be defined 
by the relation $\mathbb P \circ\pi_0^{-1} = \mu$. Then 
the map 
 $$
 L^2(\mu) \ni f \stackrel{V_0}\longrightarrow f\circ\pi_0 
 \in L^2(\mathbb P)
 $$
is an isometry. Moreover, the operator $U :L^2(\mathbb P) 
\to L^2(\mathbb P)$ defined by the formula
$$
U = V_0 S_\sigma V_0^*
$$
is an isometry.
\end{proposition}
 
\end{example}

\subsection{Markovian functions} 
In this subsection, we consider a class of functions defined
by an endomorphism $\s \in End\sms$. The following 
definition is motivated by relation \eqref{eq Markovian}.
 
\begin{definition}\label{def_Markovian} Let $\sigma$ be an 
onto endomorphism of $(X, \B, \mu)$. 
A function $\varphi\in \FXB$ satisfying 
\be\label{eq-varphi}
\int_X (f\cs) \varphi \; d\mu = \int_X f \; d\mu 
\ee
for all $f $ in $L^1(\mu)$ is called a \textit{Markovian}
function. The set of all Markovian functions is denoted by 
$M(\sigma, \mu)$. We denote  $M_2(\sigma, \mu) = M(\sigma, 
\mu) \cap L^2(\mu)$. 
\end{definition}

The Markovian functions were considered in a series of 
papers \cite{Hawkins1994}, \cite{HawkinsSilva1991}, 
\cite{DajaniHawkins1994}, and others. 

\begin{remark}\label{rem Mark}

(1) The set $M(\sigma, \mu)$ is convex. 

(2) Suppose that $\mu$ is a forward quasi-invariant 
measure for 
$\sigma \in End\sms$. Then the set of Markovian functions 
$M(\sigma, \mu)$ is not empty because 
$\omega_\mu \in M(\sigma, \mu)$ due to relation 
\eqref{eq Markovian}. 
 
(3)  For two equivalent measures $\mu$ and $\nu$,  we 
discuss 
relation between the functions $\omega_\mu$ and $\omega_\nu$
in Theorem \ref{thm om_mu and om_nu}. As above, the function 
$\omega_\mu $ generates a cocycle by setting 
$\omega_n(x) = \frac{d\mu\cs^n}{d\mu}(x)$ (where $\omega_1 = 
\omega_\mu$).

(4) Let $g$ be a function from $L^1(\mu)$. It was shown  
in \cite{Hawkins1994} that, for the measure $d\nu = gd\mu$, 
the  following holds:
$$
\int_{X} f(\sigma x) \dfrac{g\omega_\nu }{g\cs}(x)\; d\mu = 
\int_X f(x) \;d\mu,  \quad f \in  L^1(\mu).
$$
This means that the  function $\dfrac{g\omega_\nu }{g\cs}$ 
is Markovian with respect to  $(\sigma, \mu)$. 
\end{remark} 

Based on the facts from Remark \ref{rem Mark}, we prove the 
following result.

\begin{proposition}\label{prop Mark for mu nu}
Let  $\sigma \in End(X, \B, \mu)$. Suppose a measure 
$\nu$ is equivalent to $\mu$ and $g = \frac{d\nu}{d\mu}$. Then
$$
M(\sigma, \nu) = g^{-1} M(\sigma, \mu) (g\circ\sigma).
$$
\end{proposition}

\begin{proof}
For a function $\varphi \in M(\sigma, \mu)$, show that 
$h = g^{-1} \varphi (g\cs) \in M(\sigma, \nu)$. For any 
$f \in L^1(\nu)$, we have
$$
\ba
\int_X (f\cs)  g^{-1} \varphi (g\cs) \;d\nu = &
\int_X (fg)\cs \va \; d\mu\\
= & \int_X fg \; d\mu\\
=& \int_X f\; d\nu.
\ea
$$
This shows that the function $h$ is in $M(\sigma, \nu)$. 

Conversely, let $\va$ be a Markovian function from 
$M(\sigma, \nu)$. Since $d\nu = gd\mu$, 
$$
\ba \int_X fg\; d\mu = & \int_X (f\cs) \va g \; d\mu\\
=& \int_X (f\cs)(g\cs)  [(g\cs)^{-1}\va g] \; d\mu. 
\ea
$$
The latter means that $(g\cs)^{-1}\va g \in M(\sigma, \mu)$.
\end{proof}

\begin{lemma}
    Let $\va_1, ..., \va_k$ be Markovian functions from 
$M(\s, \mu)$. Then the function $\psi= (\va_k\cs^{k-1}) \cdots 
    (\va_2\cs) \va_1$ belongs to $M(\s^k, \mu)$.
\end{lemma}

\begin{proof} We compute 
$$
\ba 
\int_X (f\cs^k) \psi \; d\mu = & 
\int_X [(f\cs^{k-1})\cs  (\va_k\cs^{k-2})\cs\cdots (\va_2\cs)] 
\va_1\; d\mu  \\
 =& \int_X(f\cs^{k-1}) (\va_k\cs^{k-2}) \cdots (\va_3 \cs)\va_2\; 
 d\mu \\
 & \ \ \cdots \cdots  \\
=& \int_X f\; d\mu.
\ea
$$   
\end{proof}

\section{Operators generated by endomorphisms}
\label{sect End Operators} 

This section considers several linear operators  
naturally defined by surjective endomorphisms of $\sms$. 
These operators act in $L^2(\mu)$ and other functional 
spaces.

\subsection{Transfer operators and endomorphisms}
\label{subs TO endom}

We define a transfer operator in general settings using 
only the Borel structure of the space $(X, \B)$. Transfer 
operators are extensively studied for various dynamical 
systems applying the properties of phase spaces. 
 
\begin{definition} Let $\FXB$ be the set of all bounded 
Borel functions \footnote{In this section, we will consider 
real-valued functions for definiteness; the case of
complex-valued functions can be done similarly.}
and let $R: \FXB \to \FXB$ be a linear operator. 
Then $R$ is called a \textit{transfer operator} if it 
satisfies  the following properties:

(i) $f \geq 0 \ \Longrightarrow \ R(f) \geq 0$ (i.e., $R$ is 
a positive operator);

(ii)  for any Borel functions  $f, g \in \mathcal F(X, \B)$, 
the \textit{pull-out property} holds
\be\label{eq char property of r via sigma 1}
R((f\circ \sigma) g) = f R(g).
\ee
To emphasize that a transfer operator $R$ is defined by an
onto endomorphism $\sigma$, we will also write $R$ as 
$(R, \sigma)$.  
\end{definition}

Let $\mathbbm 1$ be a function on $(X, \B)$ that takes the 
only value 1. If $R(\mathbbm 1)(x) > 0$ for all $x\in X$, 
then we say that $R$ is a \emph{strict transfer operator}. 
If $R(\mathbbm 1) = \mathbbm 1$, then the transfer operator 
$R$ is  called \textit{normalized}.  

Every transfer operator $R$ defines an action on the 
set of probability measures $M_1(X)$: given 
$\mu \in M_1(X)$, set
$$
(\mu R)(f) = \int_X R(f)\; d\mu, \quad f \in \FXB.  
$$ 
If $\mu = \mu R$, then $\mu$ is called $R$-invariant.
A measure $\mu \in M_1(X)$ is called 
\textit{strongly invariant}
with respect to a transfer operator  $(R, \sigma)$ if $\mu$  
$\sigma$-invariant and $\mu R = \mu$. 

Restrictions of transfer operators on Banach or Hilbert 
spaces give more possibilities to study their properties. 

\begin{lemma}\label{lem:strong inv} Let $\sms$ be a 
probability standard measure space. 
Suppose that $(R, \s)$ is a normalized transfer operator 
acting in 
$L^2\sms$ where $\mu \in M_1(X)$. Then $\mu$ is strongly 
invariant with respect to $(R, \s)$ if and only if
$R^*(\mathbbm 1) = \mathbbm 1$.
\end{lemma}
\begin{proof} Here and below we denote by $\langle \cdot,
\cdot\rangle_{\mu}$ the inner product in $L^2(\mu)$. 
Since $R$ is normalized, we can write
$$
\langle f, \mathbbm 1\rangle_{\mu\csi1} =  \langle f\cs, 
R^*(\mathbbm 1) \rangle_{\mu} = \langle R(f\cs), 
\mathbbm 1 \rangle_{\mu} = \langle f R(\mathbbm 1), 
\mathbbm 1
\rangle_{\mu} = \langle f, \mathbbm 1\rangle_\mu, 
$$
that is $\mu $ is $\s$-invariant if 
$R^*(\mathbbm 1) = \mathbbm 1$.

Similarly, we see that the condition $R^*(\mathbbm 1) =1$ is
equivalent to  $\mu = \mu R$:
$$
\langle f, \mathbbm 1\rangle_{\mu R} = \langle R(f), 
\mathbbm 
1\rangle_{\mu} = \langle f, R^*(\mathbbm 1)\rangle_{\mu} =
\langle f, \mathbbm 1\rangle_{\mu}. 
$$
\end{proof}

\begin{example}\label{ex_RokhlinTO} Let $\s \in End\sms$, 
$\mu(X) =1$, and let the partition $\xi_\sigma$ of $X$ be 
defined by preimages $\{\s^{-1}(x) : x \in X\}$ of $\s$.
We give an example of a transfer operator $R_\s$ 
which is determined by the system of conditional measures 
$\{\mu_C\}$  
(see  Subsection \ref{subsect Meas part}) over the 
partition $\xi_\sigma$ of $X$.
Define a linear operator $R_\s$ acting on Borel bounded 
functions over the standard 
probability measure space  $(X, \B, \mu)$ by setting 
\be\label{eq TO via cond syst meas Intro}
R_\s(f)(x) := \int_{C_x} f(y)\; d\mu_{C_x}(y)
\ee
where  $C_x =  \sigma^{-1}(x)$. 
\end{example}

\begin{lemma}\label{lem R via cond syst meas}
The operator $R_\s: \FXB \to \FXB$ defined by 
(\ref{eq TO via cond syst meas Intro}) is a transfer 
operator. 
\end{lemma}

\begin{proof}
Clearly, $R_\s$ is a positive normalized operator. To see that 
(\ref{eq char property of r via sigma 1}) holds, we 
calculate
\begin{eqnarray*}
  R_\s((f\circ\sigma) g)(x) & = & \int_{C_x} (f\circ\sigma)
  (y)g(y)\; d\mu_{C_x}\\
  \\
   &=& f(x) \int_{C_x} g(y)\; d\mu_{C_x}(y)\\
   \\
  &=& f(x) (R_\s g)(x).
\end{eqnarray*}
Here we used the fact that $f(\sigma(y)) = f(x)$ for $y 
\in C_x = \sigma^{-1}(x)$.
\end{proof}

For an onto endomorphism $\sigma$ acting on the space 
$\sms$, we consider the subalgebra $\mc A = 
\{\sigma^{-1} (B) : B \in \B\}$ of  $\B$.
It is a well-known fact that there exists the conditional 
expectation $\mathbb E_\s : L^2(X, \B, \mu) \to L^2(X, 
\sB, \mu)$. To simplify the formulas, we will use also the 
following notations: 
$\mc A = \sB$, $L^2(\mu) =  
L^2(X, \B, \mu)$, and $L^2(\mu_{\mathcal A}) =  
L^2(X, \sB, \mu)$. Below, we will describe the
 operator $\mathbb E_\s$ explicitly in terms of the 
 composition operator $S_\s$. 
 
It turns out that, for every transfer operator $R$, one
can define another operator which is, in some sense,
analogous to the conditional expectation. 
For this, let $(R, \sigma)$ be a normalized transfer 
operator. We define  
\be\label{eq E for R}
E : \FXB \to \mc F(X, \sB) : f \mapsto R(f) \cs. 
\ee
We discuss the properties of the operator $E$ in   
Proposition \ref{prop TO - condexp}. Some of them are
proved in \cite{BezuglyiJorgensen2018(book)}.

\begin{proposition} 
\label{prop TO - condexp} Let $R$ be a normalized transfer 
operator and $E(f) = R(f)\cs$. Then the following 
properties hold:

(1)  $E$ is positive and $E^2 = E$, 

\quad\ \ $E(\FXB) = \mc F(X, \sB)$, 

\quad\ \ $E|_{\mc F(X, \sB)} = id$, 
 
\quad \ \ $R \circ E\circ R = R^2$ and $R\circ E = R$.

(2) For $S_\sigma(f) = f\cs$, we have
$$
(RS_\sigma)(f) = f, \ \ \ \ \ (S_\sigma R)(f) = E(f). 
$$ 

(3) A bounded Borel function $f$ belongs to $\mc F(X, \sB)$
if and only if there exists a function $g \in \FXB$ such 
that $f = g\cs$.
\end{proposition}

\begin{proof}
These properties are proved directly. We only check that 
$R \circ E\circ R = R^2$. Indeed, 
$$
R [ E( R(f))] = R[ R(R(f)\cs] = R^2(f) R(\mathbbm 1) =R^2(f).
$$
\end{proof} 

\begin{corollary}
Let $R$ be the transfer operator defined in 
\eqref{eq TO via cond syst meas Intro}. Then the conditional 
expectation $E : \FXB \to \mc F(X, \sB) : 
f \mapsto R(f) \cs$ acts by the formula
$$
E(f) =  R(f)\cs = \int_{\sigma^{-1}(\s (x))} f(y) \; 
d\mu_{C_{\sigma (x)}}(y).
$$
\end{corollary}

\subsection{Composition operators and Markovian functions}

We recall our notation: $\s$ is a surjective endomorphism of
a probability measure space $\sms$, $S_\s : f \mapsto 
f\cs $ is the composition operator, and $M(\s, \mu)$ is 
the set of Markovian functions.

Let $t = t(x)$ be a bounded Borel function, 
and $\mu$ a $\s$-invariant probability measure on $(X, \B)$.
We define the operator $P_t$ on $L^2(\mu)$ by setting
$$
P_t (f) = t (f\cs),\quad f \in L^2(\mu).
$$
We call the operator $P_t$ a \textit{weighted 
composition operator}. Clearly, $P_t$ is a bounded operator 
in $L^2(\mu)$.
  
\begin{theorem}\label{thm S* is TO}  
(1) For $\sigma \in End\sms$ as above, consider the 
composition 
operator $S_\sigma$ in the Hilbert space $L^2(\mu)$. 
Then the adjoint operator  $S_\sigma^*$ acts by the 
formula:
\be\label{eq formula for S*}
S_\sigma^*(g) = \frac{(gd\mu)\circ 
\sigma^{-1}}{d\mu}, \quad g\in L^2(\mu_{\mc A}).
\ee

(2) The adjoint operator $S_\sigma^*$ is a transfer operator
$$
S_\sigma^*(g(f\cs)) = f S_\sigma^*(g).
$$
The transfer operator $S_\sigma^*$ is normalized if and 
only if $\mu$ is $\sigma$-invariant.

(3) For a function $t \in \mc F(X, \B)$, the adjoint 
operator $P_t^*$ is a non-normalized transfer operator. 
\end{theorem}   

\begin{proof}
(1) For functions $f, g \in L^2(\mu)$, we have
$$
\ba 
\langle S_\s f, g\rangle_\mu = & \int_X (f\cs) g \; d\mu\\
= & \int f \; (gd\mu)\csi1\\
= & \int f \frac {(gd\mu)\csi1} {d\mu} \; d\mu \\
= & \langle f,  S_\s^*  g\rangle_\mu
\ea 
$$
which proves \eqref{eq formula for S*}. 

(2) To prove the pull-out property,  we compute, for 
arbitrary functions $f, g, h \in L^2(\mu)$, 
$$
\ba 
\int_X h S_\sigma^*((f\cs) g) \; d\mu = & 
\int_X S_\sigma(h) (f\cs) g\; d\mu\\
= & \int_X ((hf)\cs) \; g\; d\mu \\ 
= & \int_X hf S_\sigma^*(g) \; d\mu. 
\ea 
$$ 
Finally, we see that
$$
\int_X f \; d\mu\csi1=  \int_X (f\cs) \mathbbm 1 \; 
d\mu  = \int_X f S_\sigma^*(\mathbbm 1)\; d\mu, 
$$
and $S_\s^*(\mathbbm 1) = \mathbbm 1 \ 
\Longleftrightarrow \ \mu\csi1  =\mu$.

(3) The same proof as in (1) gives the formula
$$ 
P_t^*(g) = \frac{(t g d\mu)\circ 
\sigma^{-1}}{d\mu}, \quad g\in L^2(\mu_{\mc A}). 
$$
This shows that $P^*_t$ is not normalized. 

To prove that $P_t^*$ satisfies the pull-out property,
we write
$$
\ba 
\langle h, P_t^*(g(f\cs)) \rangle_\mu = & 
\langle P_t(h), g(f\cs) \rangle_\mu \\
= & \langle t (h\cs), g(f\cs) \rangle_\mu\\
= & \langle P_t(hf), g \rangle_\mu \\
= & \langle h, f P_t^*(g) \rangle_\mu. \\
\ea 
$$
\end{proof}

\begin{corollary}\label{cor cond exp is ss*}
(1) Let $\sigma \in End\sms$ be such that $\mu = \mu\csi1$.
Then the conditional expectation 
$\mathbb E_\s : L^2\sms \to L^2(X, \mc A, \mu_{\mc A})$ 
can be represented as $S_\sigma S^*_\sigma$.

(2) If $\mu$ is forward quasi-invariant, then the 
conditional expectation $\mathbb E_\s$ coincides with 
$T_\s T^*_\s$ where the isometry $T_\s (f)=\sqrt{\omega_\mu} 
(f\cs)$ is defined in Lemma \ref{lem_isom}. 
\end{corollary}

\begin{proof} 
The fact that $(S_\sigma S^*_\sigma)^2 = S_\sigma 
S_\sigma^*$ follows from the identity $S_\sigma^* 
S_\sigma = \mathbb I$. 

Next, we verify that $\langle f, S_\sigma S_\sigma^*(h)
\rangle_\mu  = \langle f, h\rangle_\mu$ for 
$h \in L^2(\mu_{\mc A})$. Recall that if $f \in 
L^2(\mu_{\mc A})$, then there 
exists $g \in L^2(\mu)$ such that $f = g\cs$. Then, 
using the fact that $S_\s^*$ is a normalized transfer 
operator (Theorem \ref{thm S* is TO}), we have
$$
\ba 
\langle g\cs, S_\sigma S_\sigma^*(h) \rangle_\mu 
= & \langle S_\s^*(g\cs), S_\sigma^*(h) \rangle_\mu \\
= &\langle g S_\s^*(\mathbbm 1), S_\sigma^*(h)\rangle_\mu\\
=& \langle g, S_\sigma^*(h) \rangle_\mu \\
= & \langle g\cs, h \rangle_\mu \\
= & \langle f, h \rangle_\mu \\
\ea 
$$
The orthogonality of the projection 
$S_\sigma S_\sigma^*$ is obtained from the relation 
$$
\langle S_\sigma S_\sigma^* f, g\cs \rangle_\mu = 
\langle S^*_\s f, g \rangle_\mu = 
\langle f, g\cs  \rangle_\mu. 
$$
It proves that $\mathbb E_\s = S_\sigma S_\sigma^*$. 

(2) The case of a quasi-invariant measure $\mu$ is 
considered similarly. We note that $\omega_\mu$ is a Borel
function measurable with respect to $\sB$. Hence, 
every function $f\in L^2(\mu_{\mc A})$ there exists a
function $g \in L^2(\mu)$ such that $f = \sqrt{\omega_\mu}
(g\cs)$. Then we repeat the above calculations.
We leave the details to the reader. 
\end{proof}

We associated with every endomorphism $\s \in End\sms$
two transfer operators $R$ and $S^*_\s$. It turns out that 
they coincide in $L^2(\mu)$.

\begin{theorem}\label{thm R = S^*}
Let $\s \in End\sms$. Then the transfer operators $R_\s$
and $S^*_\s$ coincide in $L^2(\mu)$, where $R_\s$ is   
defined in \eqref{eq TO via cond syst meas Intro} and
$S_\s^*$ satisfies \eqref{eq formula for S*}.
\end{theorem}

\begin{proof}
We compute $S^*_\s(f)$ using the disintegration of
$\mu$ with respect to the conditional measures $\mu_x$
on $C_x = \s^{-1}(x)$:
$$
\ba
\langle S(g), f\rangle_\mu = &\int_X(g\cs) f\;d\mu\\
= & \int_X\left( \int_{C_x} g(\s(y))  f(y)\;d\mu_x(y) 
\right) \; d\mu(x) \\
= & \int_X g(x) \left( \int_{C_x}  f(y)\; d\mu_x(y)
\right) \; d\mu(x) \\
=& \langle g, S^*_\s(f) \rangle_\mu
\ea
$$
From the latter, we see that 
$$
S^*_\s(f) = \int_{C_x}  f(y)\; d\mu_x(y) = R_\s(f).
$$
\end{proof}

\begin{remark}
(1) It follows from Lemma \ref{lem:strong inv} and 
Theorem \ref{thm R = S^*} that, for the 
transfer operator $S^*_\s$ in $L^2(\mu)$, the measure
$\mu$ is strong invariant if and only if it is 
$\s$-invariant.

(2) Theorem \ref{thm R = S^*} implies that $E(f) = 
R(f)\cs$ coincides with $\mathbb E_\s$ for $R = S^*_\s$.
\end{remark}

We now consider weighted composition operators  where 
the weight function is Markovian; for consistency, we will 
write $P_\va$ for such a weighted composition operator.  
We will continue discussing the properties of weighted 
composition operators in the next section. 

Let $\sigma \in End\sms$ be a surjective endomorphism, 
and $\mu$ is a forward quasi-invariant measure.  
Consider the operator 
\be\label{eq def Phi}
P_\varphi  : f \to \varphi (f\cs)
\ee
which is formally defined in $\FXB$. It can be  also 
written as 
$$
P_\va(f) = M_\varphi S_\sigma(f), \quad f \in \FXB
$$ 
where $M_\varphi$ is the multiplication operator.

\begin{lemma} Let $\sigma, \va,$ and $P_\va$ be as above.
Then $\varphi$ is a Markovian function ($\va \in 
M(\sigma, \mu)$) if and only if $\mu$ is $P_\va$-invariant, 
i.e., $\mu P_\va = \mu$ where 
$$
\mu P_\va (f) = \int_X P_\va(f) \; d\mu.
$$
In particular, $\mu = \mu P_{\omega_\mu}$.
\end{lemma}

\begin{proof}
Indeed, if $\va$ is Markovian, then 
$$
\int_X P_\va(f) \; d\mu = \int_X (f\cs) \va \; d\mu = 
\int_X f \; d\mu  
$$ 
which means that $\mu$ is $P_\va$-invariant. The 
converse statement also follows from the relation above.
\end{proof}

\begin{lemma}
Let $\va \in M(\s, \mu)$ and $P_\va$ a weighted composition
operator. Them $P_\va : L^1(\mu) \to L^1(\mu_{\mc A})$ 
and $P_{\sqva} :  L^2(\mu) \to L^2(\mu_{\mc A})$ are 
isometric operators.
\end{lemma}

\begin{proof} Straightforward. 
\end{proof}

\begin{proposition}\label{prop phi markovian} 
Let $\s \in End\sms$  and $\va$ is a function from 
$L^2(\mu)$. Then $\varphi$ is a Markovian function if and 
only if $S^*_\sigma(\varphi) = \mathbbm 1$.
\end{proposition}

\begin{proof} Let $f \in L^2(\mu)$. 
Then the result follows from the following relations:
$$`
\ba 
\int_X f\; d\mu =  & \int_X (f\cs)\varphi\; d\mu\\
= & \int S_\s(f) \varphi\; d\mu \\
=  & \int_X f S_\sigma^*(\varphi) \; d\mu.
\ea
$$
\end{proof}
  
\begin{theorem}\label{thm Phi in L2}
Let $P_\va$ be defined in $L^2(\mu)$ according to 
\eqref{eq def Phi} where $\varphi> 0$.
\noindent
(1) The following statements are equivalent:

(i)  $P_\va$ is an isometry in $L^2(\mu)$;

(ii) the composition operator $S_\sigma$ is an 
isometry in $L^2(\varphi \mu)$;

(iii)  $S^*_\sigma(\va) = \mathbbm 1$; 

(iv)  
$$
\frac{(\va  d\mu) \csi1}{d\mu} = 1 \ \ \mathrm{a.e.}
$$ 

\noindent
(2) If $\va$ is a positive Borel function from $L^2(\mu)$, 
then the adjoint  operator $P^*_\va$ is 
a transfer operator and $P^*_\va$ is normalized if 
and only if $\va$ is Markovian. 
\end{theorem}  
  
\begin{proof}
(1) The proof of the first statement  uses the arguments 
given in the proofs of Theorems  \ref{prop phi markovian} 
and \ref{thm S* is TO}. We leave the details for the  reader.

(2)   It is clear that $P^*_\va$ is positive because we have 
the following formula for $P^*_\va$:
$$
P^*_\va(g) = \frac{(\va g d\mu)\csi1}{d\mu}.
$$
Show that it satisfies the pull-out
property. Since $P_\va = M_\va S_\sigma$ and $S_\sigma^*$ 
is a transfer operator, we obtain
$$
\ba
P^*_\va ((f\cs )g) =  &\ S^*_\sigma M^*_\va[(f\cs)g]\\
= &  \ S^*_\sigma [\ol \va g (f\cs)]\\
= & \ f  S^*_\sigma (\ol \va g)\\
= & \ f  S^*_\sigma  M^*_\va (g)\\
= &\  f P^*_\va (g).
\ea
$$
To finish the proof, we note that 
$$
\ba
\int_X  (f\cs )\va \mathbbm 1 \; d\mu  = & \int_X P_\va (f) 
\mathbbm 1 \; d\mu \\
= &\ \int_X f P^*_\va(\mathbbm 1) \; d\mu.
\ea
$$
Hence 
$$
\int_X  (f\cs )\va  \; d\mu  =  \int_X f \; d\mu 
$$
 if and only if  $ P^*_\va(\mathbbm 1) = \mathbbm 1$.
  \end{proof}
  
\subsection{Radon-Nikodym derivatives and conditional 
expectations}

As above,  $\sigma \in End\sms$
is an onto endomorphism with quasi-invariant measure $\mu$. 
Equation \eqref{eq Markovian}
defines a uniquely determined Radon-Nikodym derivative 
$\omega_\mu$ which is $\sigma^{-1}(\B)$-measurable function.

Let $\mathbb E_\s$ denote the conditional expectation from
$L^2(\mu)$ onto $L^2(\mu_{\mc A})$, where $\mc A = \sB$ and 
$\mu_{\mc A}$ is the projection of $\mu$ onto the 
sigma-algebra $\mc A$. We recall that $\mathbb E_\s = 
S_\s S^*_\s$ for $\s$-invariant measure $\mu$ and
$\mathbb E_\s = T_\s T^*_\s$  for $\s$-quasi-invariant 
measure $\mu$, see Corollary \ref{cor cond exp is ss*}.

Suppose that $\nu$ is another measure on $(X, \B)$ which is 
equivalent to the measure $\mu$. In the next theorem, we 
show how Markovian functions with respect to the measures
$\mu$ and $\nu$ are related (see also
Remark \ref{rem rho for mu and nu}).
 
\begin{theorem}\label{thm om_mu and om_nu}
Let $\sigma \in End\sms$ and $d\nu(x) = h(x)d\mu(x)$ 
where $h(x) > 0$  $\mu$-a.e. Let $\psi \in M(\s, \nu)$
and $\va\in M(\s, \mu)$. Then  
\be\label{eq omega_nu}
\mathbb E_\s(h) \psi = \va   (h\cs).
\ee
Hence, relation \eqref{eq omega_nu} establishes a one-to-one 
correspondence between the sets of Markovian functions  
$M(\s, \mu)$ and $M(\s, \nu)$.
 \end{theorem}
 
\begin{proof} Since the both sides of \eqref{eq omega_nu} 
are measurable with respect to $\sB$, it suffices to prove 
that, for every function $g\in \FXB$,
\be\label{eq omega_nu 2}
 \int_X (g\cs) \mathbb E_\s(h) \; \psi \; d\mu =  
 \int_X (g\cs) (h \cs)\;  \va \; d\mu. 
 \ee
We compute the left-hand side and the right-hand side in 
\eqref{eq omega_nu 2} separately using the definition of 
Markovian functions and the properties of conditional 
expectations. For the RHS:
 $$
 \ba 
 \int_X (g\cs) (h \cs)\;  \va \; d\mu = &
 \int_X (gh)\cs \;\va \; d\mu\\
 = & \int_X gh \; d\mu\\
 = & \int_X g \; d\nu.
 \ea 
 $$
For the LHS, we use the fact that $\mathbb E_\s$ is the 
conditional  expectation and the function $(g\cs) \psi$ is 
$\sB$-measurable: 
 $$
 \ba 
  \int_X (g\cs) \mathbb E_\s(h) \;\psi \; d\mu = & 
  \int_X (g\cs) h \;\psi \; d\mu\\
  = & \int_X (g\cs)  \;\psi \; d\nu\\
  =& \int_X g \; d\nu.
 \ea
 $$
 \end{proof}
 
Let $R = (R, \s)$ be a transfer operator where 
$\s$ is an onto endomorphism of $\sms$. 
Recall that we have defined in \eqref{eq E for R} the 
operator $E = R(f) \cs: \FXB \to  \mc F(X, \sB)$ which 
is an analog of the conditional expectation 
$\mathbb E_\s$. In the following statement, we find out 
under what conditions on the measure $\mu$ and the transfer 
operator $R$ the operator $E$ coincides with the genuine 
conditional expectation $\mathbb E_\s$.
 
\begin{theorem} \label{thm E vs E*}
In the setting formulated above, the operator 
$E = R(f)\cs$ coincides with $\mathbb E_\s = S_\sigma 
S^*_\sigma$ in $L^2(\mu)$ if and only if  
\be\label{eq equiv cond exp}
R(f) \frac{d\mu \csi1}{d\mu} = \frac{(fd\mu )\csi1}{d\mu}, 
\quad f\in L^2(\mu).
\ee
Relation \eqref{eq equiv cond exp} can be also 
written as $\rho_\mu R(f) = S^*_\sigma(f)$. 
\end{theorem} 

\begin{proof}
It was shown in Proposition \ref{prop TO - condexp} that 
$E^2 = E$. It remains to find out under what conditions 
the relation $E = E^*$ hold. Clearly, it is  
equivalent to the property
\be\label{eq_cond-exp-eq}
\int_X g E(f)\; d\mu = \int_X gf \; d\mu, \quad \forall g 
\in \mc F(X, \sB). 
\ee
Representing $g $ as $h\cs$ ($h\in \FXB$), we obtain that 
\eqref{eq_cond-exp-eq} is equivalent to 
$$
\int_X (h\cs) R(f)\cs\; d\mu =   \int_X (h\cs) f\; d\mu 
$$ 
or 
$$
\int_X h R(f) \; d\mu\csi1 =   \int_X h \; (fd\mu) \csi1.
$$
This means that 
$$
R(f) = \frac{(fd\mu) \csi1}{d\mu \csi1}
$$
which is equivalent to \eqref{eq equiv cond exp}.
This proves that $E = \mathbb E_\s$ if and only if 
$S^*_\s(f) = \rho_\mu R(f)$. 

We note that if $\mu$ is a $\sigma$-invariant measure, then 
$$
R(f) = \frac{(fd\mu) \csi1}{d\mu}
$$
which coincides with $S^*_\sigma(f)$ by 
\eqref{eq formula for S*}. 
It follows then that $\mathbb E_\s = E = R(f)\cs$ in 
the case of $\sigma$-invariant measure $\mu$. 
Moreover, it is obvious that the condition $R(f) = 
S^*_\sigma(f)$ implies 
the invariance of $\mu$ with respect to $\sigma$.
\end{proof}

\begin{corollary}
In notation given above, the operator $E_{\rho_\mu} = 
(\rho_\mu R)\cs$ coincides with $\mathbb E_\s$. 
\end{corollary}

\begin{proof}
We noted that $\rho_\mu R$ is a transfer operator coinciding 
with $S^*_\sigma$, and therefore we can 
define $E_{\rho_\mu} = S_\sigma R_{\rho_\mu}$. By Corollary 
\ref{cor cond exp is ss*}, we obtain that $S_\sigma 
R_{\rho_\mu}= S_\sigma S^*_\sigma = \mathbb E_\s$ which 
proves the statement.
\end{proof}

\begin{remark}\label{rem prop mark fn}
The results of Proposition \ref{prop Mark for mu nu} and
Theorem \ref{thm om_mu and om_nu} can be interpreted 
as follows.

Let $G$ denote the group $\mc F_+(X, \B)$ of Borel 
bounded strictly positive functions. Then $G$ acts on the 
set $\{M(\sigma,\nu):\nu \sim \mu\}$. This action $\alpha =
\{\alpha_f : f \in G\}$ is defined by the rule:
$$
\alpha_f(\va) =  \frac{f\cs}{f} \va,\ \quad  \va \in 
M(\sigma, \mu).
$$
Clearly,
$$
\alpha_f(M(\sigma, \mu)) = M(\sigma, \nu),
$$
where $d\nu = fd\mu$, and $\alpha_f\alpha_g = \alpha_{fg}$. 
The action $\alpha$ is free in the sense that it satisfies 
the property $M(\sigma, \mu) \cap M(\sigma, \nu) = \emptyset$
if $\nu \sim \mu$. Moreover, $\alpha$ is transitive. 

\end{remark}


\section{Cuntz relations for invariant and 
quasi-invariant measures} \label{sect q-inv}

Starting with the measurable category, and disintegration of 
the appropriate measures, we showed above that careful 
choice of Hilbert spaces allows for a powerful tool in the 
analysis of endomorphisms and branching systems (in the 
measurable setting). In more detail, the steps from 
transformations in measure space to $L^2$ spaces and 
operators are often called ``passing to the Koopman 
operators''. In our context, 
the non-commutativity for the operators under consideration 
is captured well with the Cuntz relations, or rather their 
representations; see \cite{Cuntz1977}, 
\cite{JorgensenTian2020}, \cite{Bezuglyi_Jorgensen_2015}.
Recall that, following J. Cuntz, for every $n$, one 
introduces a $C^*$-algebra $\mathcal O_{|\Lambda|}$ defined by a 
system of $|\Lambda|$ generators $T_i$. These generators may be 
realized as operators in Hilbert space, say $H$ as follows: 
The relations (Cuntz-relations) state that the $T_i$ system 
is represented by isometries with orthogonal ranges in 
$H$ such that the sum of these ranges is $H$. (Think of the 
subspaces as sub-bands.) In other words, via the isometries, 
$H$ arises as an orthogonal sum of copies of itself. 
As a $C^*$-algebra, $\mathcal O_{|\Lambda|}$ is simple. Its 
representations are important, and they play a crucial role 
in the study of self-similar dynamics and self-similar 
geometries.

\subsection{Quasi-invariant measure}\label{subs q-inv meas}
Let $\s \in End\sms$ where $\mu$ is a forward and backward 
quasi-invariant measure. We recall that, in this case, the 
corresponding Radon-Nikodym derivatives $\omega_\mu$ and 
$\rho_\mu$ are well-defined functions satisfying 
\eqref{eq Markovian} and \eqref{eq RN rho}.  
In the remaining sections of this paper, we will consider 
$L^2$-spaces of complex-valued functions.

Let $\va$ be a positive Markovian function, $\va \in 
M(\s, \mu)$.
Then we define a \textit{weighted composition operator} 
$S_{\va}$ acting on $L^2(\mu)$ by
\be\label{eq S_va}
S_{\va} (f) = \sqrt{\varphi} (f \cs)
\ee
Equivalently, $S_\va = M_{\sqva} S_\sigma$ where $M_{\sqva}$
denotes the operator of multiplication, and $S_\s$ is the 
composition operator. 

We mention two important particular cases of \eqref{eq S_va} 
when (a) $\va = \omega_\mu$ and (b) $\va = \mathbbm 1$. Case (b) 
occurs if and only if $\mu\csi1 = d\mu$.

\begin{lemma}
(1)  Let $\va$ be a Markovian function from $M_2(\s, \mu)
= M(\sigma, \mu) \cap L^2(\mu)$, and
$S_\s$ the composition operator. Then $S_\s^*(\va) = 
\mathbbm 1$. 

(2) The function  $\va$ is Markovian with respect to 
$\sigma$ 
and $\mu$ if and only if $(\va d\mu)\csi1 = \mu$.
\end{lemma}
\begin{proof}
The first statement follows from the definition of a 
Markovian function:
$$
\int_V f\; d\mu = \int_X (f\cs) \va\; d\mu = 
\int_X S_\s(f) \va \; d\mu= \int_V f S_\s^*(\va) \; d\mu.
$$

The second statement is a reformulation of relation 
\eqref{eq-varphi}.
\end{proof}

\begin{lemma}\label{lem S_{va} iso}
    The operator $S_{\va}$ is an isometry in $L^2(\mu)$.
\end{lemma}

\begin{proof} It follows from \eqref{eq-varphi} that the 
function $\sqva (f\cs) \in L^2(\mu)$ if $f \in L^2(\mu)$. 
Since $\va$ is Markovian, we have 
$$
\ba
\langle S_{\va}(f), S_{\va}(g)\rangle_\mu = & 
\int_X \sqva (f\cs) \sqva (\ol g\cs)\; d\mu \\
= & \int_X (f\ol g\cs ) \va \; d\mu \\
= & \int_X f\ol g\; d\mu, \qquad f, g \in L^2(\mu).
\ea
$$
\end{proof}

\begin{theorem}\label{thm Sphi is TO}  
(1) For $\sigma \in End\sms$ as above and the operator $S_{\va}$,
the adjoint operator  $S_\sigma^*$ acts by the formula:
\be\label{eq adj of S}
S_\va^*(g) = \frac{(g \sqva d\mu)\circ \sigma^{-1}}{d\mu}, 
\quad g\in L^2(\mu).
\ee
 (2) The adjoint operator $S_\va^*$ is a transfer operator 
satisfying the pull-out property:
\be\label{eq pullout}
S_\va^*(g(f\cs)) = f S_\va^*(g).
\ee
The operator $S_\va^*$ is normalized if and only if $\mu$ is 
$\sigma$-invariant.
\end{theorem}  

\begin{proof} 
(1) For functions $f, g \in L^2(\mu)$, we have
$$
\int_X \sqva (f\cs) g \; d\mu = \int f \; (\sqva gd\mu)\csi1 =  
\int f \frac {(gd\mu)\csi1}{d\mu} \; d\mu 
$$
which proves \eqref{eq adj of S}. 

(2) The operator $S^*_\va$ is obviously positive. 
To prove the pull-out property,  we compute, for arbitrary 
functions $f, g, h \in L^2(\mu)$, 
$$
\ba
\int_X h S_\va^*((f\cs) g) \; d\mu = & \int_X S_\va(h) (f\cs) 
g\; d\mu \\ 
= & \int_X \sqva (hf)\cs \; g\; d\mu \\
= & \int_X S_\va (hf) g \; d\mu\\
= & \int_X hf S^*_\va (g) \; d\mu.
\ea
$$ 
This proves that \eqref{eq pullout} holds.

We see that
$$
S_\va^*(\mathbbm 1) = \frac{(\sqva d\mu)\circ \sigma^{-1}}
{d\mu}.
$$
Hence, $S_\va^*$ is normalized if and only if $\mu = \mu\csi1$ 
and $\va = \mathbbm 1$. This can be proved as follows
$$
\int_X f \; d\mu\csi1=  \int_X (f\cs) \mathbbm 1 \; d\mu  = 
\int_X f S_\sigma^*(\mathbbm 1)\; d\mu, 
$$
and $S_{\s}^*(\mathbbm 1) = \mathbbm 1 \ \Longleftrightarrow \ 
\mu\csi1  =\mu$.
\end{proof}

\begin{remark}
    One can easily check that, for $\va = \omega_\mu$, 
$$
S^*_{\omega_\mu} (\frac{1}{\sqrt{\omega_\mu}}) = \rho_\mu.
$$
\end{remark}

We recall that $\mc A $ denotes the subalgebra $\s^{-1}(\B)$ 
and $\mu_{\mc A}$ denotes the restriction of $\mu$ onto $\mc A$.
It is an important observation that a function $f$ is 
$\mc A$-measurable if and only if there exists a $\B$-measurable
function $g$ such that $f = g\cs$. 

For a fixed Markovian function $\va$, consider the subspace 
$\mc H_\va$ of function spanned by $\sqva$ and $\mc A$-measurable 
functions:
$$
\mc H_\va = \{ \sqva (f\cs) : f \in L^2(\mu)\}
$$

\begin{proposition}\label{prop cond exp}
Let $\s, \va,$ and $S_\va$ be as above. Then 

(1) $\mathbb E_\va: = S_\va S^*_\va$ is
an orthogonal projection from $L^2(\mu)$ onto $\mc H_\va$;

(2) $\mathbb E_\va ((f\cs) g) = (f\cs) \mathbb E_\va (g)$

(3) $S^*_\va (f \mathbb E_\va (g)) = 
S^*_\va (g) S^*_\va (f\sqva)$.
\end{proposition}

\begin{proof} 
(1) Let $f, g \in L^2(\mu)$. Then we write
$$
\ba
\langle S_\va S^*_\va f, \sqva (g\cs)\rangle_\mu = & 
\langle S_\va S^*_\va f, S_\va g\rangle_\mu \\
= & \langle  S^*_\va f, g \rangle_\mu \\
= & \langle   f, S_\va g \rangle_\mu \\
= & \langle   f, \sqva  (g\cs) \rangle_\mu.
\ea
$$
Hence, $(f - S_\va S^*_\va f) \perp \mc H_\va $.    

For (2), we first note that $S_\va (fg) = \sqva (g\cs) (f\cs) =
S_\va(g) S_\s(f)$, and then
$$
\ba
\mathbb E_\va ((f\cs) g) =  & S_\va S^*_\va ((f\cs) g)\\
=& S_\va ( S^*_\va(g) f)\\
= &  (S_\va S^*_\va)(g) S_\s (f) \\
= & (f\cs) \mathbb E_\va (g).
\ea 
$$

For (3), we use the pull-out property of $S^*_\va$:
$$
\ba 
S^*_\va (f \mathbb E_\va (g)) = & S^*_\va (f S_\va S^*_\va (g))\\
= & S^*_\va (f \sqva (S^*_\va (g)\cs )) \\
=& S^*_\va (g) S^*_\va (f \sqva). 
\ea
$$

\end{proof}
  
Let $\{m_i : i \in \Lambda\}$ be a collection of complex-valued
functions from $L^2(\mu)$. We fix a Markovian function $\va$. 
For every $i \in \Lambda$, we define
\be\label{eq def T}
T_{m_i} (f) = m_i \sqva (f\cs) = M_{m_i} S_\va(f)
\ee
where $M_{m_i} $ is the multiplication operator. 
Then $T_{m_i}$ is an operator acting from $L^2(\mu)$ onto
$m_i \mc H_\va$.

\begin{lemma}
The operator $T_m(f) = m \sqva (f\cs)$ is bounded on $L^2(\mu)$
if and only if $|m|^2 \in L^\infty(\mu)$.
\end{lemma}

\begin{proof}
For $f \in L^2(\mu)$, we have
$$
\ba
||T_m(f)||^2 = & \int_X (m\sqva (f\cs)) (\ol m  \sqva (\ol f\cs)
\; d\mu\\
=& \int_X |m|^2 (|f|^2\cs) \va \; d\mu\\
\leq & \sup |m|^2  \int_X (|f|^2\cs) \va \; d\mu\\
= & \sup |m|^2 ||f||^2.
\ea
$$
\end{proof}

\begin{lemma}\label{lem T_m iso}
The operator $T_{m}$ is an isometry in $L^2(\mu)$ if and only if 
$$
S^*_\va (\sqva |m|^2) = \mathbbm 1.
$$
\end{lemma}

\begin{proof}
We note that $T^*_{m}(f) = S^*_\va M_{\ol m} (f)$. Using 
the pull-out property for $S^*_\va$, we can write 
$$
T^*_{m} T_{m}(f) =  S^*_\va M_{\ol m} (m_i \sqva (f\cs)) =
f S^*_\va (|m|^2 \sqva)
$$
where  $f \in L^2(\mu)$. This proves the lemma. 
\end{proof}

\begin{remark}
It follows from Lemma \ref{lem T_m iso} that $T_{m}$ is 
an isometry if and only if 
$$
\mathbb E_\va (\sqva |m|^2 ) = \sqva.
$$
\end{remark}
     
\begin{lemma}\label{lem T*T}
 The operators $T_m, S_\va,$ and $\mathbb E_\va$ satisfy the 
 properties:
\be\label{eq T*T}
T^*_{m_1}T_{m_2}(f) = f S^*_\va(\sqva\; \ol m_1 m_2), 
\ee
\be \label{eq TT*}
T_{m_1}T^*_{m_2}(f) = m_1 \mathbb E_\va (\ol m_2 f).
\ee 
\end{lemma}

\begin{proof}
Indeed, we have 
$$
\ba 
T^*_{m_1}T_{m_2}(f) =  & T^*_{m_1} (\sqva\; m_2 (f\cs))\\
= & S^*_{m_2} M_{\ol m_1} (\sqva\; m_2 (f\cs))\\
= & f S^*_{m_2} (\sqva\; \ol m_1 m_2),
\ea 
$$
and
$$
T_{m_1}T^*_{m_2}(f) = M_{m_1} S_\va S^*_\va M_{\ol m_2}(f) 
= m_1 \mathbb E_\va (\ol m_2 f).
$$
\end{proof}

\begin{remark}
Let $g$ be a bounded positive Borel function. Take a Markovian 
function $\va \in M(\sigma, \mu)$ and consider $\psi =
(g\cs) \va g^{-1}$. By Proposition \ref{prop Mark for mu nu} we 
see that $\psi \in M(\s, \nu)$ where $d\nu = g d\mu$.  
Then 
$$
S_\psi = M^{-1}_{\sqrt{g}} S_\va M_{\sqrt{g}}.
$$
Then a direct computation shows that $S_\psi$ is an isometry 
in  $L^2(\nu)$ where $d\nu = g\, d\mu$.

Denote by $\wt T_m$ the operator acting on $L^2(\nu)$:  
$f \mapsto m S_\psi(f)$. By Lemma \ref{lem T_m iso}, $\wt T_m$
is an isometry in $L^2(\nu)$ if and only if 
$S^*_\psi(\sqrt{\psi}|m|^2) = \mathbbm 1$. It follows from the 
definition of the operators $T_m$ and $\wt T_m$ that 
$$
\wt T_m = M^{-1}_{\sqrt{g}} T_m M_{\sqrt{g}}.
$$
\end{remark}

\begin{theorem}\label{thm main1} Let $\sigma$ be an onto 
endomorphism of $\sms$ where $\mu$ is quasi-invariant with 
respect to $\s$. 
Let $\{m_i : i \in \Lambda\}$ be a family of
complex-valued functions.  
The operators $\{T_{m_i}, i \in \Lambda$\} generate  a 
representation of the Cuntz algebra $\mathcal O_{|\Lambda|}$ 
if and only if 
$$
(i)\ \ S^*_\va (\sqva \; \ol m_j m_i) = \delta_{ij} 
\mathbbm 1,
$$
$$
(ii) \ \ \sum_{i\in \Lambda} m_i \mathbb E_\va( \ol m_i f) 
= f,
\ \ \ f \in L^2(\mu).
$$ 
\end{theorem}

\begin{proof} We first note that condition (i) of the 
theorem 
implies that the operators $T_{m_i}$ are isometries because
of Lemma \ref{lem T_m iso}. 
The Cuntz relations for isometries  $\{T_{m_i}\}_{i\in 
\Lambda}$ mean that 
$$
\sum_{i\in \Lambda} T_{m_i} T_{m_i}^* = \mathbb I
$$
and the projections $T_{m_i} T_{m_i}^*$ are mutually 
orthogonal.

We will show that, for $f,g \in L^2(\mu)$ and $i\neq j$, 
the vectors $T_{m_i} T_{m_i}^*(f)$ and $T_{m_j} 
T_{m_j}^*(g)$ 
are orthogonal if and only if condition (i) of the theorem 
holds. In the following computation, we use \eqref{eq T*T}. 
$$
\ba 
\langle T_{m_i} T_{m_i}^*(f), T_{m_j}T_{m_j}^*(g)\rangle_\mu 
= & \langle (T_{m_j}^* T_{m_i}) T_{m_i}^*(f)), 
T_{m_j}^*(g)\rangle_\mu \\
= & \langle S^*_\va(\sqva\; \ol m_j m_i)(\xi),\eta 
\rangle_\mu
\ea 
$$
where $\xi = T_{m_i}^*(f)$ and $\eta = T_{m_j}^*(g)$. Since 
$f, g$ are arbitrary functions, the left-hand side is zero 
if and only if $S^*_\va(\sqva\; \ol m_j m_i) = \delta_{ij} 
\mathbb I$.

Next, we use \eqref{eq TT*} to see when the identity 
operator 
is decomposed in the sum of  orthogonal projections 
$T_{m_i}T_{m_i}^*$. It follows from Lemma \ref{lem T*T} that,
for any $f \in L^2(\mu)$,
$$
\sum_{i\in \Lambda} T_{m_i}T_{m_i}^*(f) = 
\sum_{i\in \Lambda} m_i \mathbb E_\va (\ol m_i f) 
$$
Hence, the property $\sum_{i\in \Lambda} T_{m_i} T_{m_i}^* = 
\mathbb I$ is equivalent to condition (ii) of the  theorem.
\end{proof}

\begin{corollary}
It follows from Theorem \ref{thm main1} that 
$$
L^2(\mu) = \bigoplus_{i\in \Lambda} \ m_i \mc H_\va.
$$
In other words, $\{m_i : i \in \Lambda\}$ is an orthogonal 
module basis for $L^2(\mu)$ over $\mc H_\va =
\sqva\; L^2(\mu_{\mc A})$ where $\mc A = \s^{-1} (\B)$. 
\end{corollary}

\subsection{Invariant measure} 

In what follows we will consider the case of a $\s$-invariant
measure $\mu$. 

\begin{lemma}
Let $\s \in End\sms$ be an onto endomorphism. Then if 
$\mu $ is $\s$-invariant probability measure, then the only 
$\sB$-measurable Markovian 
function is the constant function $\mathbbm 1$.
\end{lemma} 

\begin{proof}
Indeed, let $\va$ be a Markovian function. 
Then we have
$$
\int_X (f\cs) \va \; d\mu = \int_X f\; d\mu = \int_X f \; 
d(\mu \cs)^{-1} = \int_X (f\cs) \; d\mu.
$$
Since $f$ is arbitrary, we get $\va = \mathbbm 1$ on $\sB$.
\end{proof}

We will apply the constructions of 
Subsection \ref{subs q-inv meas} to the case when 
$\va =\mathbbm 1$. Because the results of Subsection 
\ref{subs q-inv meas} are proved in more general settings, 
we will just formulate the corresponding facts without proof.  

Let $\underline m = \{m_i : i \in \Lambda\}$ be a collection 
of complex-valued bounded functions, and let 
$$
S_{m_i}(f) = m_i (f\cs), \quad f \in L^2(\mu)
$$
be the corresponding weighted composition operators. 

In this lemma, we collect the properties of $S_{m_i}$.

\begin{lemma} Let $f, g \in L^2(\mu)$.

(1) $S^*_m$ is an isometry on $L^2(\mu)$ if and only if
$S^*_\s( |m|^2) = \mathbbm 1$.

(2) $\mathbb E_\s := S_\s S^*_\s$ is the conditional expectation
from $L^2(\mu)$ onto $L^2(\mu_{\mc A})$. 

(3) $E_\s((f\cs)g) = (f\cs) \mathbb E_\s(g)$. 

(4) $S^*_\s (f \mathbb E_\s(g)) = S^*_\s (f) S^*_\s (g)$.

(5) $S_{m_1}^* S_{m_2}(f) = S^*_\s(\ol m_1 m_2) f$.

(6) $S_{m_1} S_{m_2}^*(f) = m_1 \mathbb E(\ol m_2 f)$.
\end{lemma}

Here is the modified version of Theorem \ref{thm main1}.
 
\begin{theorem}\label{thm main} 
Let $\sigma$ be a surjective endomorphism of $\sms$ with 
$\s$-invariant measure $\mu$.
Let $\{m_i: i \in \Lambda\}$ be a family of bounded 
complex-valued functions.  
The operators $\{S_{m_i}, i \in \Lambda$\} generate  a 
representation of the Cuntz algebra $\mathcal O_{|\Lambda|}$ 
if and only if 
$$
(i)\ \ S^*_\s ( \ol m_j m_i) = \delta_{ij} 
\mathbbm 1,
$$
$$
(ii) \ \ \sum_{i\in \Lambda} m_i \mathbb E_\s( \ol m_i f) = f,
\ \ \ f \in L^2(\mu).
$$ 
\end{theorem}

It follows from Theorem \ref{thm main} that the Hilbert space
$L^2(\mu)$ admits the decomposition 
$$
L^2(\mu) = \bigoplus_{i\in \Lambda} \ m_i L^2(\mu_{\mc A})
$$
if and only if the operators $S_{m_i}$ are the generators 
of a representation of the Cuntz algebra. 
In other words, $\{m_i : i \in \Lambda\}$ is an orthogonal 
module basis for $L^2(\mu)$ over $L^2(\mu_{\mc A})$ where 
$\mc A = \s^{-1} (\B)$. 
\\

We finish this section with a discussion of the relations 
between the operators $S_\s$, $S_\va$ and $\mathbb E_\s$,
$\mathbb E_\va$ where $\va$ is a Markovian function.

\begin{proposition}\label{prop sigma vs va}
The following formulas hold: for $g\in L^2(\mu)$
$$
(1)\ \ S^*_\va (g) = S^*_\s(\sqva\; g),\quad 
(2)\ \ \mathbb E_\va(g) = \sqva \; \mathbb E_\s (\sqva\; g).
$$
\end{proposition}

\begin{proof}
For (1), take any functions $f, g \in L^2(\mu)$ and 
compute
$$
\ba 
\langle f, S^*_\va(g) \rangle_\mu = & 
\langle S^*_\va(f), g \rangle_\mu \\
= & \langle\sqva\;  (f\cs), g \rangle_\mu\\
=& \langle S_\s(f), \sqva\; g\rangle_\mu \\
=& \langle f, S^*_\s(\sqva\; g) \rangle_\mu. 
\ea
$$

To see that (2) is true, we recall that $S_\va (f) = 
\sqva\; S_\s (f)$, and then we can write
$$
\mathbb E_\va(f) = S_\va S^*_\va (f) = 
S_\va S^*_\s(\sqva \; f) = \sqva\; S_\s S^*_\s(\sqva \; f)
= \sqva\; \mathbb E_\s(\sqva \; f).
$$
\end{proof}

The results of Proposition \ref{prop sigma vs va} will be 
used in the next section.

\section{The set of wavelet filters} 
\label{sect wavelet filters}
In this section, we 
answer the question about the structure of the set of 
\textit{wavelet filters}, i.e., we describe the set 
of functions satisfying the conditions of 
Theorem \ref{thm main1}. 
We consider two different approaches: (i) it will be shown 
that there the set $\mathfrak M$ is isomorphic to the so 
called loop group $\mc G$; 
(ii) using the decomposition into cyclic representations, 
we describe elements $\underline m \in \mathfrak M$ as the
collection of cyclic vectors for the representation of 
$L^\infty(\mc A)$ in $L^2\sms$.

\subsection{Actions of loop groups on wavelet filters}

Consider bounded operators $B(l^2(\Lambda))$ in the Hilbert 
space $l^2(\Lambda)$ where $\Lambda$ is a countably infinite 
set. If $\{e_i : i \in \Lambda\}$ is the canonical 
orthonormal basis in $l^2(\Lambda)$, then we define the 
infinite $|\Lambda| \times |\Lambda|$ matrix $\wh A = 
(a_{ij})$ by setting $a_{ij} = \langle e_i, A e_j \rangle$.
This observation will allow us to work with matrix notation
in the computations below.

Let $\mc U$ be the group of all unitary operators in 
$B(l^2(\Lambda))$. Denote by $\mc G$ the group of Borel
functions on $(X, \B)$ with values in $\mc U$ (we recall 
that $\mc U$ is a Polish group). We use the notation
$G = (g_{ij}(x))$ for elements of $\mc G$. Then every 
entry $g_{ij}(x)$ is a Borel complex-valued matrix. 
The group $\mc G$ is called the \textit{loop group}.
We remark that for $G \in \mc G$, $G^* G = \mathbb I$ where
the matrix $G^* = \ol G^T$. We will use the relation
\be\label{eq G unitary}
\sum_{l\in \Lambda} \ol g_{li} g_{lj} = \delta_{ij}
\ee
where $G = (g_{ij}) \in \mc G$.

Let $\va$ be a Markovian function. We recall that, in this 
case, the operator $S_\va$ is isometric. 
We  denote the set of generalized wavelet filters by 
\be\label{eq M_phi}
\mathfrak M_\va := \{\underline m = (m_i)_{i \in \Lambda} : 
S_\va ^*(\sqva\; m_i \ol m_j) = \delta_{ij} \mathbbm 1,\ \ \ 
\sum_{i\in \Lambda} 
m_i \mathbb E_\va (\ol m_j f) = f\}.
\ee
We consider simultaneously the case of $\sigma$-invariant 
measure $\mu$ and the corresponding operators $S_\s$ and
$\mathbb E_\s$, see Section \ref{sect q-inv} for 
properties of these operators. In this case, we use the set
\be \label{eq M_sigma}
\mathfrak M_\s := \{\underline m = (m_i)_{i \in \Lambda} : 
S_\s ^*( m_i  \ol m_j) = \delta_{ij} \mathbbm 1,\ \ \ 
\sum_{i\in \Lambda} m_i \mathbb E_\s (\ol m_j f) = f\}.
\ee 

It turns out that the group $\mc G$ acts on the sets 
$\mathfrak M_\va$ and $\mathfrak M_\s$. Indeed, 
for a fixed $\underline m \in \mathfrak M_\va$ (or 
$\underline m \in \mathfrak M_\s$) and $G \in \mc 
G$, we define 
$\underline m^G := (m^G_i(x) : i \in \Lambda)$ by the 
formula 
\be \label{eq action G}
m^G_i(x)  = \sum_{j \in \Lambda} (\ol g_{ji}\cs)(x) m_j(x) 
\ee 
or $\underline m^G = (G^*\cs) \underline m$ in a short form. 

In the next statements, we will study the properties of 
this action of $\mc G$ on the sets of wavelet filters. 
For definiteness, we formulate these results for the 
set $\mathfrak M_\va$. The same proofs work for the 
action of $\mc G$ on $\mathfrak M_\s$, we will omit them.
We will show in the next lemmas that: (i) the set 
$\mathfrak M_\va$ is invariant with respect to the action
of group $\mc G$; (ii) formula
\eqref{eq action G} defines a group action on 
$\mathfrak M_\va$; (iii) this action of $\mc G$ is free and
transitive. 

\begin{lemma}\label{lem m^G}
If $\underline m \in \mathfrak M_\va$, then $\underline m^G 
\in \mathfrak M_\va$.
\end{lemma}

\begin{proof}
We will verify that the family of functions $\underline m^G=
(m_i^G(x) : i \in \Lambda)$ satisfies 
conditions (i) and (ii) of Theorem \ref{thm main1}. 
For (i), we use \eqref{eq G unitary} and the fact that 
$S^*_\va$ is a transfer operator satisfying the pull-out 
property: 
$$
\ba 
S^*_\va (\sqva \; m^G_i \ol m^G_j ) = & 
S^*_\va \left(\sqva \sum_{k \in \Lambda} (\ol g_{ki}\cs) m_k 
\cdot \sum_{l \in \Lambda} \ol{(\ol g_{jl}\cs)  m_l} 
\right)\\
= & \sum_{k \in \Lambda} \sum_{l \in \Lambda} \ol 
g_{ki}g_{lj} S^*_\va (\sqva \ol m_l m_k)\\ 
= & \sum_{k \in \Lambda} \sum_{l \in \Lambda} 
\ol g_{ki}g_{lj} \delta_{kl}\\ 
=& \sum_{k \in \Lambda} \ol g_{ki}  g_{kj} \\
=& \delta_{ij}.
\ea 
$$
For (ii), let $f \in L^2(\mu)$, then
$$
\ba
\sum_{i\in \Lambda} m_i^G \mathbb E_\va (\ol m^G_i f) = &
\sum_{i\in \Lambda} \sum_{k \in \Lambda} (\ol g_{ki}\cs) 
m_k \mathbb E_\va\left(\sum_{l \in \Lambda} (g_{li}\cs) 
\ol m_lf\right)\\
= & \sum_{i\in \Lambda} \sum_{k \in \Lambda}\sum_{l \in 
\Lambda} (\ol g_{ki}\cs) ( g_{li}\cs) m_k \mathbb 
E_\va (\ol m_l f)\\  =&  \sum_{i\in \Lambda} \sum_{k \in 
\Lambda}\sum_{l \in \Lambda}
((\ol g_{ki} g_{li}) \cs)  m_k \mathbb E_\va (\ol m_l f)\\
=&  \sum_{i\in \Lambda} \sum_{k \in \Lambda} \delta_{kl} 
m_k \mathbb E_\va (\ol m_l f)\\
=&  \sum_{k \in \Lambda} m_k \mathbb E_\va (\ol m_l f)\\
= & f. 
\ea
$$
We used here the equality $GG^* = \mathbb I$ or 
$\sum_i \ol g_{ki} g_{li} = \delta_{kl}$.
\end{proof}

In the next lemma, we show that \eqref{eq action G} defines
an action of the group $\mc G$ on $\mathfrak M$.

\begin{lemma}\label{lem GH}
For every $\underline m \in \mathfrak M_\va$ and every 
$G, H \in \mc G$, 
\be \label{eq GH}
(\underline m^G)^H = \underline {m}^{GH}, \quad 
\underline m^{\mathbb I} = \underline m.
\ee
\end{lemma}

\begin{proof}
Indeed, if $\mathbb I$ is the identity matrix, then  
$\underline  m^{\mathbb I} = \underline m$ for every 
$\underline m \in \mathfrak M_\va$. 

Let $G,H \in \mathfrak M_\va$. Show that \eqref{eq GH} 
holds:
$$
\ba 
(m^G)^H_i =& \sum_{j\in \Lambda} (\ol h_{ji} \cs) 
m_j^G \\
= & \sum_{j\in \Lambda} (\ol h_{ji} \cs) 
\sum_{k\in \Lambda} (\ol g_{kj} \cs) m_k\\
= & \sum_{k\in \Lambda} \sum_{j\in \Lambda} ((\ol g_{ij} 
\ol h_{jk})\cs)m_k \\
= & \sum_{k\in \Lambda} (\ol f_{ik}\cs)  m_k\\
=& m^{GH}_i,
\ea 
$$
where the entries of $GH$ are denoted by $(f_{ik})$.
\end{proof}

\begin{lemma}\label{lem iso of actions}
Let $\va$ be a positive Markovian function. 

(1) The map $\Phi :\underline m \mapsto \sqva\; 
\underline m$ 
defines an isomorphism between the sets $\mathfrak M_\va$ 
and $\mathfrak M_\s$.

(2) The map $\Phi$ implements the conjugation of actions
of $\mc G$ on the sets $\mathfrak M_\s$ and 
$\mathfrak M_\va$.
\end{lemma}

\begin{proof}
(1) We will show that if $\underline m$ satisfies 
\eqref{eq M_phi}, then $\sqva \underline m $ belongs to
$\mathfrak M_\s$, i.e., \eqref{eq M_sigma} holds.
For this, we check that
$$
S^*_\s(\sqva m_i \; \sqva \ol m_j) = S^*_\va(\sqva\;
m_i\ol m_j) = \delta_{ij}, \ \ i,j \in \Lambda,
$$
and, using Proposition \ref{prop sigma vs va},  
$$
\ba 
\sum_{i\in\Lambda} \sqva m_i \mathbb E_\s(\sqva \ol m_i f) 
=& \sum_{i\in\Lambda}\sqva m_i S_\s S^*_\s(\sqva\ol m_if)\\
= & \sum_{i\in\Lambda} m_i S_\va S^*_\va (\ol m_i f)\\
= & f.
\ea 
$$

(2) It can be checked directly that, for every $G\in \mc G$
and $\underline m \in \mathfrak M_\va$,
$$
\Phi \underline m^G = (\Phi \underline m)^G.
$$

\end{proof}

\begin{lemma}\label{lem transitive}
Let $\va$ be a positive Markovian function. 
The action of $\mc G$ on $\mathfrak M_\va$  defined in 
\eqref{eq action G} is free and transitive.
\end{lemma}

\begin{proof} It follows from Lemma \ref{lem iso of actions}
that it suffices to show that the action of $\mc G$ 
on $\mathfrak M_\s$ is free and transitive.

Let $\underline m = (m_i)$ and 
$\underline n = (n_j)$ be two elements of
the set $\mathfrak M_\s$. Define an infinite matrix $G$ by 
setting
$$
g_{ij} = S^*_\va(m_i \ol n_j), \quad i, j \in \Lambda.
$$
We show that $\underline m^G = \underline n$ to prove that 
the action is transitive. Indeed, 
$$
\ba 
m_i^G = & \sum_{k \in \Lambda} (\ol g_{k,i}\cs) m_k \\
= & \sum_{k \in \Lambda} (S^*_\s (\ol m_k n_i)\cs ) m_k\\
=& \sum_{k \in \Lambda} S_\s S^*_\s (\ol m_k n_i) m_k\\
= & \sum_{k \in \Lambda} \mathbb E_\s(\ol m_k n_i) m_k\\
= & n_i.
\ea 
$$
We used here relation \eqref{eq M_sigma}. 

To see that this action is free, we assume that 
$\underline m^G = \underline m$ for some
$\underline m \in \mathfrak M_\s$ and $G\in \mc G$. Then
$$
\sum_{k \in \Lambda} (\ol g_{ki}\cs) m_k = m_i, \quad i \in 
\Lambda.
$$
Multiply both sides by $\ol m_j$ and apply the operator 
$S^*_\s$: 
$$
\sum_{k \in \Lambda} S^*_\s((\ol g_{ki}\cs) m_k \ol m_j) =
S^*_\s(m_i \ol m_j)
$$
Use the pull-out property and \eqref{eq M_sigma}:
$$
\sum_{k \in \Lambda} \ol g_{ki}S^*_\s(m_k \ol m_j) =
\delta_{ij}.
$$
Hence, 
$$
\sum_{k \in \Lambda} \ol g_{ki} \delta_{kj} = \delta_{ij} 
$$
and $\ol g_{li} = \delta_{ij}$. This proves that 
$G = \mathbb I$ and the action is free.

\end{proof}

The following theorem gives a complete description of the 
set  $\mathfrak M_\va$. This result immediately 
follows from the proven Lemmas \ref{lem m^G} - \ref{lem 
transitive}.

\begin{theorem}\label{thm action} 
(1) The set $\mathfrak M_\va$ is isomorphic (as a set) to the
loop group $\mc G$. 

(2) For every element $G = (g_{ij})$ of the loop group, 
there exists a wavelet filter $\underline m$ such that 
$g_{ij} = S^*_{\va}(\sqva\; m_i \ol m_j)$.
\end{theorem}

\subsection{Endomorphisms, wavelet filters and cyclic 
representations}

In this subsection, we will use another approach to describe the
set $\mathfrak M$ of wavelet filters associated with an 
endomorphism $\s$ of $\sms$. We will assume here that $\s$ 
is measure-preserving. 

Let $\mc A= \sB$ and $\mu_{\mc A}$ is the restriction of $\mu$
onto the sigma-subalgebra $\mc A$. Denote by $\mathfrak A$ the
set $L^{\infty}(\mu_{\mc A}) $ of bounded $\sB$-measurable 
functions. The $*$-algebra $\mathfrak A$ acts on $L^2(\mu)$ by
multiplication operators: for every $f \in L^2(\mu)$ and 
$\gamma \in \mathfrak A$
$$
f \longmapsto \gamma f = M_\gamma (f).
$$
This formula defines a representation $\pi$ of $\mathfrak A$ 
in $L^2(\mu)$. Using the theorem about the decomposition of 
non-degenerate representations into orthogonal cyclic 
representations, we can write 
\be\label{eq cyclic decomp}
L^2(\mu) = \bigoplus_{i \in \Lambda} \mc H_i, \quad  
\pi = \bigoplus_{i \in \Lambda} \pi_i
\ee
where $\pi_i$ is a cyclic representation and $\mc H_i$ is 
the closure  of $\{\pi_i(\gamma) h_i : \gamma \in \mathfrak 
A\}$. The vectors $h_i, i \in \Lambda$, are called cyclic. 
The set of indexes $\Lambda$ is countably infinite. 

We will use the cyclic vectors $(h_i)$ to construct a 
wavelet filter $\underline m = (m_i)$ satisfying the 
conditions of Theorem \ref{thm main}.  

\begin{lemma}
There are functions $(g_i)$ such that the cyclic vectors 
$(m_i = h_i (g_i\cs))$ for representations 
of $\mathfrak A$ have the property 
\be \label{eq:m_i^2} 
S_\s^*(|m_i|^2) = \mathbbm 1.
\ee
\end{lemma}

\begin{proof} We note that for every function $\gamma 
\in \mathfrak A$ there exists a function $g \in 
L^\infty(\mu)$ 
such that $\gamma = g\cs$. Hence, if $h_i$ is a cyclic vector
for  a representation of $\mathfrak A$, then $m_i$ is cyclic,
too.

We will determine the function $g_i$ such that  $m_i = h_i 
(g_i\cs)$ satisfies the statement of the lemma. For this, we 
compute 
$$
S^*_\s(|m_i|^2)  = S^*_\s(|h_i (g_i\cs)|^2)  =  |g_i|^2  
S^*_\s(|h_i|^2),
$$
and take $g_i$ such that the above expression equals 1 a.e.
\end{proof}

For functions $f \in L^2(\mu)$, we define the operators 
\be\label{eq S_i}
S_i(f) = m_i (f\cs) = M_{m_i} S_\s(f).
\ee

\begin{lemma}\label{lem:orthogonality}
Let the functions $(m_i)$ satisfy \eqref{eq:m_i^2}. Then the 
operators $S_i, i \in \Lambda$, are isometries satisfying 
the property 
\be\label{eq:S*m_im_j}
S_\s^*(m_i \ol m_j) = \delta_{ij}, \quad i,j \in \Lambda. 
\ee
\end{lemma}

\begin{proof}
We first note that $S^*_i = S^*_\s M_{\ol m_i}$, and 
$$
S^*_iS_i = \mathbb I \ \ \Longleftrightarrow \ \ 
S_\s^*(|m_i|^2) = \mathbbm 1.
$$

Next, we use the fact that cyclic representations, which are 
defined by the distinct vectors $m_i$ and $m_j$, are 
mutually 
orthogonal. Applying the fact that the transfer operator 
$S^*_\s$ has the pull-out property, we compute, for any 
$\gamma_1, \gamma_2 \in \mathfrak A$, 
$$
\ba
\langle \gamma_1 m_i, \gamma_2 m_j\rangle_\mu = & 
\langle (g_1\cs)\; m_i, (g_2 \cs)\;  m_j\rangle_\mu \\
=& \langle (g_1\cs)m_i \ol m_j, S_\s(g_2) \rangle_\mu \\
=& \langle S^*_\s((g_1\cs)m_i \ol m_j), g_2 \rangle_\mu \\
= & \langle g_1 S^*_\s(m_i \ol m_j), g_2 \rangle_\mu.
\ea
$$
Hence, the inner product is zero if and only if 
\eqref{eq:S*m_im_j} holds.
\end{proof}

\begin{theorem}\label{thm main 2}
Let $(m_i : i\in \Lambda)$ be a set of cyclic vectors 
satisfying 
\eqref{eq:m_i^2}. Then $\underline m = (m_i)$ is a wavelet 
filter if and only if $\sum_{i \in \Lambda} S^*_iS_i = 
\mathbb I$. In other words, $\underline m \in \mathfrak M$ if
and only if the operators $S_i$ are the generators of a 
representation of the Cuntz algebra $\mathcal O_{|\Lambda|}$.
\end{theorem}

\begin{proof} To prove the theorem, we use the result of
Theorem \eqref{thm main}. For this, we should check that 
$\{S_i : i\in \Lambda\}$ are the operators satisfying conditions
(i) and (ii) of Theorem \eqref{thm main}. Condition (i) is 
proved in Lemma \eqref{lem:orthogonality}. To verify (ii), we 
observe that, for any $f\in L^2(\mu)$ and $i\in \Lambda$,
$$
\mathbb E_\s (f\ol m_i)m_i = S_iS_i^*(f). 
$$
Indeed,
$$
\ba
\mathbb E_\s (f\ol m_i)m_i = S_iS_i^*(f) = & 
S_\s S_\s^*(f\ol m_i)m_i = S_iS_i^*(f)\\
= & S_i( S_\s^* M_{\ol m_i}(f))\\
= &  S_i S_i^*(f). 
 \ea 
$$
It follows that $\sum_{i \in \Lambda} S^*_iS_i = 
\mathbb I$ if and only if condition (ii) holds. This proves 
the theorem.
\end{proof}

\begin{corollary} Let $\s \in End\sms$ where $\mu$ is a 
probability measure. 
Let $(m_i : i\in \Lambda)$ be a set of cyclic vectors 
satisfying \eqref{eq cyclic decomp}  and \eqref{eq:m_i^2}.
Then, for a.e. $x \in X$, 
$$
\sum_{i\in \Lambda} |m_i(x)|^2 < \infty.
$$
\end{corollary}

\begin{proof}
The result follows from the fact that 
$$
\sum_{i\in \Lambda} ||m_i||_{\mu}^2 = 1.
$$
Indeed, this relation is obtained when the function 
$\mathbbm 1(x)$ is substituted into the cyclic 
decomposition \eqref{eq cyclic decomp}. 
\end{proof}

\ignore{
Let $\alpha$ be an $*$-endomorphism of a separable Hilbert 
space $\mc H$. It is well known \cite{} that there exists a 
sequence $(S_i : i \in \Lambda)$ of isometries of $\mc H$ 
with orthogonal ranges such that for every $T \in B(\mc H)$
$$
\alpha(T) = \sum_{i\in \Lambda} S_i T S_i^*.
$$
As proved in the operators $(S_i)$ are the generators of a 
representation of the Cuntz algebra $\mc O_{|\Lambda|}$.
}
\bibliographystyle{alpha}
\bibliography{EndomorphismsRef,bibliography-TO}

\newcommand{\etalchar}[1]{$^{#1}$}
\def\ocirc#1{\ifmmode\setbox0=\hbox{$#1$}\dimen0=\ht0 \advance\dimen0
  by1pt\rlap{\hbox to\wd0{\hss\raise\dimen0
  \hbox{\hskip.2em$\scriptscriptstyle\circ$}\hss}}#1\else {\accent"17 #1}\fi}
  \def\ocirc#1{\ifmmode\setbox0=\hbox{$#1$}\dimen0=\ht0 \advance\dimen0
  by1pt\rlap{\hbox to\wd0{\hss\raise\dimen0
  \hbox{\hskip.2em$\scriptscriptstyle\circ$}\hss}}#1\else {\accent"17 #1}\fi}
\begin{thebibliography}{BMPR12}

\bibitem[ACSS22]{AlpayColombo2022}
Daniel Alpay, Fabrizio Colombo, Irene Sabadini, and Baruch Schneider.
\newblock Beurling-{L}ax type theorems and {C}untz relations.
\newblock {\em Linear Algebra Appl.}, 633:152--212, 2022.

\bibitem[AJL17]{AlpayJorgensenLewkowicz2017}
Daniel Alpay, Palle Jorgensen, and Izchak Lewkowicz.
\newblock Characterizations of families of rectangular, finite impulse
  response, para-unitary systems.
\newblock {\em J. Appl. Math. Comput.}, 54(1-2):395--423, 2017.

\bibitem[AJL18]{AlpayJorgensenLewkowicz2018}
Daniel Alpay, Palle Jorgensen, and Izchak Lewkowicz.
\newblock {$W$}-{M}arkov measures, transfer operators, wavelets and
  multiresolutions.
\newblock In {\em Frames and harmonic analysis}, volume 706 of {\em Contemp.
  Math.}, pages 293--343. Amer. Math. Soc., [Providence], RI, [2018] \copyright
  2018.

\bibitem[And22]{Andrianov2022}
P.~A. Andrianov.
\newblock Multidimensional periodic discrete wavelets.
\newblock {\em Int. J. Wavelets Multiresolut. Inf. Process.}, 20(2):Paper No.
  2150053, 20, 2022.

\bibitem[BD22]{BhatDar2022}
M.~Younus Bhat and Aamir~H. Dar.
\newblock Fractional vector-valued nonuniform {MRA} and associated wavelet
  packets on {$L^2 (\Bbb {R},\Bbb {C}^M)$}.
\newblock {\em Fract. Calc. Appl. Anal.}, 25(2):687--719, 2022.

\bibitem[B{\'e}n96]{Beneteau1996}
Catherine B{\'e}n{\'e}teau.
\newblock A natural extension of a nonsingular endomorphism of a measure space.
\newblock {\em Rocky Mountain J. Math.}, 26(4):1261--1273, 1996.

\bibitem[BH09]{BruinHawkins2009}
Henk Bruin and Jane Hawkins.
\newblock Rigidity of smooth one-sided {B}ernoulli endomorphisms.
\newblock {\em New York J. Math.}, 15:451--483, 2009.

\bibitem[BJ97a]{BratteliJorgensen_1997}
O.~Bratteli and P.~E.~T. Jorgensen.
\newblock Endomorphisms of {${\mathcal B}({\mathcal H})$}. {II}. {F}initely
  correlated states on {${\mathcal O}_n$}.
\newblock {\em J. Funct. Anal.}, 145(2):323--373, 1997.

\bibitem[BJ97b]{BratteliJorgensen(1997)}
Ola Bratteli and Palle E.~T. Jorgensen.
\newblock A connection between multiresolution wavelet theory of scale {$N$}
  and representations of the {C}untz algebra {$\mathcal O_N$}.
\newblock In {\em Operator algebras and quantum field theory ({R}ome, 1996)},
  pages 151--163. Int. Press, Cambridge, MA, 1997.

\bibitem[BJ97c]{BratteliJorgensen1997}
Ola Bratteli and Palle E.~T. Jorgensen.
\newblock Isometries, shifts, {C}untz algebras and multiresolution wavelet
  analysis of scale {$N$}.
\newblock {\em Integral Equations Operator Theory}, 28(4):382--443, 1997.

\bibitem[BJ99]{BratteliJorgensen_1999}
Ola Bratteli and Palle E.~T. Jorgensen.
\newblock Iterated function systems and permutation representations of the
  {C}untz algebra.
\newblock {\em Mem. Amer. Math. Soc.}, 139(663):x+89, 1999.

\bibitem[BJ15]{Bezuglyi_Jorgensen_2015}
Sergey Bezuglyi and Palle E.~T. Jorgensen.
\newblock Representations of {C}untz-{K}rieger relations, dynamics on
  {B}ratteli diagrams, and path-space measures.
\newblock In {\em Trends in harmonic analysis and its applications}, volume 650
  of {\em Contemp. Math.}, pages 57--88. Amer. Math. Soc., Providence, RI,
  2015.

\bibitem[BJ18]{BezuglyiJorgensen2018(book)}
Sergey Bezuglyi and Palle E.~T. Jorgensen.
\newblock {\em Transfer operators, endomorphisms, and measurable partitions},
  volume 2217 of {\em Lecture Notes in Mathematics}.
\newblock Springer, Cham, 2018.

\bibitem[BLP{\etalchar{+}}10]{Baggett_et_al2010}
Lawrence~W. Baggett, Nadia~S. Larsen, Judith~A. Packer, Iain Raeburn, and Arlan
  Ramsay.
\newblock Direct limits, multiresolution analyses, and wavelets.
\newblock {\em J. Funct. Anal.}, 258(8):2714--2738, 2010.

\bibitem[BMPR12]{BaggettMerrillPackerRamsay2012}
Lawrence~W. Baggett, Kathy~D. Merrill, Judith~A. Packer, and Arlan~B. Ramsay.
\newblock Probability measures on solenoids corresponding to fractal wavelets.
\newblock {\em Trans. Amer. Math. Soc.}, 364(5):2723--2748, 2012.

\bibitem[Bog07]{Bogachev2007}
Vladimir~I. Bogachev.
\newblock {\em Measure theory. {V}ol. {I}, {II}}.
\newblock Springer-Verlag, Berlin, 2007.

\bibitem[Bru22]{Bruin(book)2022}
Henk Bruin.
\newblock {\em Topological and ergodic theory of symbolic dynamics}, volume 228
  of {\em Graduate Studies in Mathematics}.
\newblock American Mathematical Society, Providence, RI, [2022] \copyright
  2022.

\bibitem[CD22]{ChristoffersonDutkay2022}
Nicholas~J. Christoffersen and Dorin~Ervin Dutkay.
\newblock Representations of {C}untz algebras associated to random walks on
  graphs.
\newblock {\em J. Operator Theory}, 88(1):139--170, 2022.

\bibitem[CFS82]{CornfeldFominSinai1982}
Isaac~P. Cornfeld, Sergei~V. Fomin, and Yakov~G. Sina{\u\i}.
\newblock {\em Ergodic theory}, volume 245 of {\em Grundlehren der
  Mathematischen Wissenschaften [Fundamental Principles of Mathematical
  Sciences]}.
\newblock Springer-Verlag, New York, 1982.
\newblock Translated from the Russian by A. B. Sosinski{\u\i}.

\bibitem[Cun77]{Cuntz1977}
Joachim Cuntz.
\newblock Simple {$C\sp*$}-algebras generated by isometries.
\newblock {\em Comm. Math. Phys.}, 57(2):173--185, 1977.

\bibitem[DH94]{DajaniHawkins1994}
Karma~G. Dajani and Jane~M. Hawkins.
\newblock Examples of natural extensions of nonsingular endomorphisms.
\newblock {\em Proc. Amer. Math. Soc.}, 120(4):1211--1217, 1994.

\bibitem[DJ06]{DutkayJorgensen2006}
Dorin~Ervin Dutkay and Palle E.~T. Jorgensen.
\newblock Hilbert spaces built on a similarity and on dynamical
  renormalization.
\newblock {\em J. Math. Phys.}, 47(5):053504, 20, 2006.

\bibitem[DJ07]{DutkayJorgensen_2007}
Dorin~Ervin Dutkay and Palle E.~T. Jorgensen.
\newblock Martingales, endomorphisms, and covariant systems of operators in
  {H}ilbert space.
\newblock {\em J. Operator Theory}, 58(2):269--310, 2007.

\bibitem[DJ14a]{DutkayJorgensen_2014}
Dorin~Ervin Dutkay and Palle E.~T. Jorgensen.
\newblock Monic representations of the {C}untz algebra and {M}arkov measures.
\newblock {\em J. Funct. Anal.}, 267(4):1011--1034, 2014.

\bibitem[DJ14b]{DutkayJorgensen2014}
Dorin~Ervin Dutkay and Palle E.~T. Jorgensen.
\newblock The role of transfer operators and shifts in the study of fractals:
  encoding-models, analysis and geometry, commutative and non-commutative.
\newblock In {\em Geometry and analysis of fractals}, volume~88 of {\em
  Springer Proc. Math. Stat.}, pages 65--95. Springer, Heidelberg, 2014.

\bibitem[DJ15]{DutkayJorgensen2015}
Dorin~Ervin Dutkay and Palle E.~T. Jorgensen.
\newblock Representations of {C}untz algebras associated to quasi-stationary
  {M}arkov measures.
\newblock {\em Ergodic Theory Dynam. Systems}, 35(7):2080--2093, 2015.

\bibitem[DJK94]{DoughertyJacksonKechris1994}
R.~Dougherty, S.~Jackson, and A.~S. Kechris.
\newblock The structure of hyperfinite {B}orel equivalence relations.
\newblock {\em Trans. Amer. Math. Soc.}, 341(1):193--225, 1994.

\bibitem[DJS12]{DutkayJorgensenSilvestrov2012}
Dorin~Ervin Dutkay, Palle E.~T. Jorgensen, and Sergei Silvestrov.
\newblock Decomposition of wavelet representations and {M}artin boundaries.
\newblock {\em J. Funct. Anal.}, 262(3):1043--1061, 2012.

\bibitem[EFHN15]{Eisner(book)2015}
Tanja Eisner, B\'{a}lint Farkas, Markus Haase, and Rainer Nagel.
\newblock {\em Operator theoretic aspects of ergodic theory}, volume 272 of
  {\em Graduate Texts in Mathematics}.
\newblock Springer, Cham, 2015.

\bibitem[Fab87]{Fabec1987}
Raymond~C. Fabec.
\newblock Induced group actions, representations and fibered skew product
  extensions.
\newblock {\em Trans. Amer. Math. Soc.}, 301(2):489--513, 1987.

\bibitem[Fab00]{Fabec2000}
Raymond~C. Fabec.
\newblock {\em Fundamentals of infinite dimensional representation theory},
  volume 114 of {\em Chapman \& Hall/CRC Monographs and Surveys in Pure and
  Applied Mathematics}.
\newblock Chapman \& Hall/CRC, Boca Raton, FL, 2000.

\bibitem[FS22]{FengSimon2022}
De-Jun Feng and K\'{a}roly Simon.
\newblock Dimension estimates for {$C^1$} iterated function systems and
  repellers. {P}art {II}.
\newblock {\em Ergodic Theory Dynam. Systems}, 42(11):3357--3392, 2022.

\bibitem[Haw94]{Hawkins1994}
Jane~M. Hawkins.
\newblock Amenable relations for endomorphisms.
\newblock {\em Trans. Amer. Math. Soc.}, 343(1):169--191, 1994.

\bibitem[Haw21]{Hawkins2021}
Jane Hawkins.
\newblock {\em Ergodic dynamics---from basic theory to applications}, volume
  289 of {\em Graduate Texts in Mathematics}.
\newblock Springer, Cham, [2021] \copyright 2021.

\bibitem[HS91]{HawkinsSilva1991}
Jane~M. Hawkins and Cesar~E. Silva.
\newblock Noninvertible transformations admitting no absolutely continuous
  {$\sigma$}-finite invariant measure.
\newblock {\em Proc. Amer. Math. Soc.}, 111(2):455--463, 1991.

\bibitem[JKS07]{JorgensenKornelsonShuman2007}
Palle E.~T. Jorgensen, Keri Kornelson, and Karen Shuman.
\newblock Harmonic analysis of iterated function systems with overlap.
\newblock {\em J. Math. Phys.}, 48(8):083511, 35, 2007.

\bibitem[JKS11]{JorgenssenKornelsonShuman2011}
Palle E.~T. Jorgensen, Keri~A. Kornelson, and Karen~L. Shuman.
\newblock Iterated function systems, moments, and transformations of infinite
  matrices.
\newblock {\em Mem. Amer. Math. Soc.}, 213(1003):x+105, 2011.

\bibitem[Jor96]{Jorgensen1996}
Palle E.~T. Jorgensen.
\newblock A duality for endomorphisms of von {N}eumann algebras.
\newblock {\em J. Math. Phys.}, 37(3):1521--1538, 1996.

\bibitem[Jor01]{Jorgensen2001}
Palle E.~T. Jorgensen.
\newblock Ruelle operators: functions which are harmonic with respect to a
  transfer operator.
\newblock {\em Mem. Amer. Math. Soc.}, 152(720):viii+60, 2001.

\bibitem[Jor06]{Jorgensen2006(book)}
Palle E.~T. Jorgensen.
\newblock {\em Analysis and probability: wavelets, signals, fractals}, volume
  234 of {\em Graduate Texts in Mathematics}.
\newblock Springer, New York, 2006.

\bibitem[Jor18]{Jorgensen2018}
Palle E.~T. Jorgensen.
\newblock {\em Harmonic analysis}, volume 128 of {\em CBMS Regional Conference
  Series in Mathematics}.
\newblock American Mathematical Society, Providence, RI, 2018.
\newblock Smooth and non-smooth, Published for the Conference Board of the
  Mathematical Sciences.

\bibitem[JP11]{JorgensenPaolucci2011}
P.~E.~T. Jorgensen and A.~M. Paolucci.
\newblock States on the {C}untz algebras and {$p$}-adic random walks.
\newblock {\em J. Aust. Math. Soc.}, 90(2):197--211, 2011.

\bibitem[JS18]{JorgensenSong2018}
Palle E.~T. Jorgensen and Myung-Sin Song.
\newblock Markov chains and generalized wavelet multiresolutions.
\newblock {\em J. Anal.}, 26(2):259--283, 2018.

\bibitem[JT17]{JorgensenTian2017}
P.~Jorgensen and F.~Tian.
\newblock Transfer operators, induced probability spaces, and random walk
  models.
\newblock {\em Markov Process. Related Fields}, 23(2):187--210, 2017.

\bibitem[JT19]{JorgensenTian2019}
Palle Jorgensen and Feng Tian.
\newblock Dynamical properties of endomorphisms, multiresolutions, similarity
  and orthogonality relations.
\newblock {\em Discrete Contin. Dyn. Syst. Ser. S}, 12(8):2307--2348, 2019.

\bibitem[JT20]{JorgensenTian2020}
Palle Jorgensen and James Tian.
\newblock Noncommutative boundaries arising in dynamics and representations of
  the {C}untz relations.
\newblock {\em Numer. Funct. Anal. Optim.}, 41(5):571--620, 2020.

\bibitem[Kec95]{Kechris1995}
Alexander~S. Kechris.
\newblock {\em Classical descriptive set theory}, volume 156 of {\em Graduate
  Texts in Mathematics}.
\newblock Springer-Verlag, New York, 1995.

\bibitem[MV23]{MedhiViswanathan2023}
R.~Medhi and P.~Viswanathan.
\newblock On the code space and {H}utchinson measure for countable iterated
  function system consisting of cyclic {$\phi$}-contractions.
\newblock {\em Chaos Solitons Fractals}, 167:Paper No. 113011, 2023.

\bibitem[PR22]{PickloRyan2022}
Matthew~J. Picklo and Jennifer~K. Ryan.
\newblock Enhanced multiresolution analysis for multidimensional data utilizing
  line filtering techniques.
\newblock {\em SIAM J. Sci. Comput.}, 44(4):A2628--A2650, 2022.

\bibitem[PU10]{PrzytyckiUrbanski2010}
Feliks Przytycki and Mariusz Urba{\'n}ski.
\newblock {\em Conformal fractals: ergodic theory methods}, volume 371 of {\em
  London Mathematical Society Lecture Note Series}.
\newblock Cambridge University Press, Cambridge, 2010.

\bibitem[Roh49a]{Rohlin_1949}
V.~A. Rohlin.
\newblock Selected topics from the metric theory of dynamical systems.
\newblock {\em Uspehi Matem. Nauk (N.S.)}, 4(2(30)):57--128, 1949.

\bibitem[Roh49b]{Rohlin1949}
Vladimir~A. Rohlin.
\newblock On the fundamental ideas of measure theory.
\newblock {\em Mat. Sbornik N.S.}, 25(67):107--150, 1949.

\bibitem[Roh61]{Rohlin1961}
Vladimir~A. Rohlin.
\newblock Exact endomorphisms of a {L}ebesgue space.
\newblock {\em Izv. Akad. Nauk SSSR Ser. Mat.}, 25:499--530, 1961.

\bibitem[RR22]{Roychowdhury-2_2022}
Lakshmi Roychowdhury and Mrinal~Kanti Roychowdhury.
\newblock Quantization for a probability distribution generated by an infinite
  iterated function system.
\newblock {\em Commun. Korean Math. Soc.}, 37(3):765--800, 2022.

\bibitem[Sil88]{Silva1988}
Cesar~E. Silva.
\newblock On {$\mu$}-recurrent nonsingular endomorphisms.
\newblock {\em Israel J. Math.}, 61(1):1--13, 1988.

\bibitem[Sim12]{Simmons2012}
David Simmons.
\newblock Conditional measures and conditional expectation; {R}ohlin's
  disintegration theorem.
\newblock {\em Discrete Contin. Dyn. Syst.}, 32(7):2565--2582, 2012.

\bibitem[URM22]{Urbanski(book)2022}
Mariusz Urba\'{n}ski, Mario Roy, and Sara Munday.
\newblock {\em Non-invertible dynamical systems. {V}ol. 1. {E}rgodic
  theory---finite and infinite, thermodynamic formalism, symbolic dynamics and
  distance expanding maps}, volume~69 of {\em De Gruyter Expositions in
  Mathematics}.
\newblock De Gruyter, Berlin, [2022] \copyright 2022.

\end{thebibliography}

\end{document}